\documentclass{amsart}

\usepackage{amsmath}
\usepackage{amssymb}
\usepackage{commath}
\usepackage{mathtools}
\usepackage{xypic}

\usepackage{hyperref}

\theoremstyle{plain}
\newtheorem{theorem}{Theorem}
\newtheorem{lemma}[theorem]{Lemma}
\newtheorem{proposition}[theorem]{Proposition}
\newtheorem{corollary}[theorem]{Corollary}
\theoremstyle{definition}
\newtheorem{definition}[theorem]{Definition}
\newtheorem{remark}[theorem]{Remark}
\newtheorem{example}[theorem]{Example}

\numberwithin{theorem}{section}
\numberwithin{equation}{section}

\newcommand{\B}{\mathbb{B}}
\newcommand{\C}{\mathbb{C}}
\newcommand{\N}{\mathbb{N}}
\newcommand{\R}{\mathbb{R}}
\newcommand{\T}{\mathbb{T}}
\newcommand{\Z}{\mathbb{Z}}

\newcommand{\GL}{\mathrm{GL}}

\newcommand{\SU}{\mathrm{SU}}

\newcommand{\Ad}{\mathrm{Ad}}

\newcommand{\im}{\mathrm{Im}}
\newcommand{\re}{\mathrm{Re}}

\newcommand{\fg}{\mathfrak{g}}
\newcommand{\fh}{\mathfrak{h}}

\newcommand{\cT}{\mathcal{T}}
\newcommand{\cA}{\mathcal{A}}

\begin{document}

\title[Moment maps and Toeplitz operators]{Moment maps of Abelian groups and commuting Toeplitz operators acting on the unit ball}

\author{Ra\'ul Quiroga-Barranco}
\address{Centro de Investigaci\'on en Matem\'aticas, Guanajuato, M\'exico}
\email{quiroga@cimat.mx}

\author{Armando S\'anchez-Nungaray}
\address{Facultad de Matem\'aticas, Universidad Veracruzana, Veracruz, M\'exico}
\email{sancheznungaray@gmail.com}

\maketitle

\begin{abstract}
    We prove that to every connected Abelian subgroup $H$ of the biholomorphisms of the unit ball $\mathbb{B}^n$ we can associate a set of bounded symbols whose corresponding Toeplitz operators generate a commutative $C^*$-algebra on every weighted Bergman space. These symbols are of the form $a(z) = f(\mu^H(z))$, where $\mu^H$ is the moment map for the action of $H$ on $\B^n$. We show that, for this construction, if $H$ is a maximal Abelian subgroup, then the symbols introduced are precisely the $H$-invariant symbols. We provide the explicit computation of moment maps to obtain special sets of symbols described in terms of coordinates. In particular, it is proved that our symbol sets have as particular cases all symbol sets from the current literature that yield Toeplitz operators generating commutative $C^*$-algebras on all weighted Bergman spaces on the unit ball $\mathbb{B}^n$. Furthermore, we exhibit examples that show that some of the symbol sets introduced in this work have not been considered before. Finally, several explicit formulas for the corresponding spectra of the Toeplitz operators are presented. These include spectral integral expressions that simplify the known formulas for maximal Abelian subgroups for the unit ball.
\end{abstract}

%\setcounter{tocdepth}{1}
%\tableofcontents

\section{Introduction}
Toeplitz operators with special symbols on the unit ball and other bounded symmetric domains have proved to be a very interesting topic of study in operator theory. The introduction of geometric and Lie theoretic techniques has been particularly useful. This is not surprising since bounded symmetric domains, and hence the unit ball, carry rich dynamics through their biholomorphisms. Among the first works to make use of these geometric techniques we find \cite{QVUnitBall1} and \cite{QVUnitBall2}. In these references, it was shown that every maximal Abelian subgroup, or MASG for short, of biholomorphisms of the unit ball $\B^n$ yields a commutative $C^*$-algebra generated by Toeplitz operators on every weighted Bergman space.

We now give a more precise statement of the last claim. Up to conjugacy of Lie groups, there are exactly $n+2$ MASGs of the biholomorphisms of $\B^n$. We list all these groups and their actions in Section~\ref{sec:MASG_momentmaps}, where we make use of the Siegel domain $D_n$ realization of the unit ball for $n+1$ of the conjugacy classes. Hence, from now we will denote by $D$ either $\B^n$ or $D_n$. For every such MASG $G$ in that list it is considered in \cite{QVUnitBall1} and \cite{QVUnitBall2} the space of essentially bounded $G$-invariant symbols, that we will denote by $L^\infty(D)^G$ in this work. Then, it is proved in \cite{QVUnitBall1} that the $C^*$-algebra generated by Toeplitz operators with symbols in $L^\infty(D)^G$ is commutative in every weighted Bergman space. However, it turns out that if $H$ is a proper subgroup of the MASG $G$, then the Toeplitz operators with $H$-invariant symbols generate a $C^*$-algebra which is no longer commutative (see \cite{DOQJFA} for examples).

In other words, for an Abelian subgroup $H$ of the biholomorphisms of $D$, when we use $H$-invariant symbols, the group $H$ yields a commutative $C^*$-algebra generated by Toeplitz operators if and only if $H$ is maximal, where the maximality is taken by the inclusion in the collection of Abelian subgroups. The following question arises: is there a way to associate commutative $C^*$-algebras generated by Toeplitz operators to every Abelian subgroup of the biholomorphisms of $D$? This would have to involve some construction different from just taking invariant symbols.

One of the goals of this work is to give an affirmative answer to the previous question. There is indeed a geometric construction that allows us to assign to every connected Abelian subgroup $H$ of the biholomorphisms of $\B^n$ a set of symbols whose corresponding Toeplitz operators generate a commutative $C^*$-algebra on every weighted Bergman space of $D$. This geometric construction involves symplectic geometry and the so-called moment map of a symplectic action.

Since $D$ is a K\"ahler manifold for the Bergman metric, it carries a closed $2$-form that turns it into a symplectic manifold. Hence, we can associate to every smooth function $f$ on $D$ a so-called Hamiltonian vector field $X_f$. This process can be reversed under some circumstances. More precisely, for a suitable vector field $X$ it is possible to find a function $f$ such that $X_f = X$, i.e.~the Hamiltonian vector field of $f$ is $X$. Taking this one step further, if $H$ is a subgroup of biholomorphisms of $D$, then a standard differentiation process realizes the Lie subalgebra $\fh$ of $H$ as a space of vector fields on $D$. And, under some conditions, for every vector field $X$ in this space it is possible to obtain a function $f$ such that $X_f = X$. Collecting all these functions through coordinates in $\fh$ we obtain the moment map of the action of $H$. We refer to Section~\ref{sec:unitball_geometry} for the definitions and basic properties which we write down explicitly for both the unit ball $\B^n$ and the Siegel domain $D_n$.

For the time being let us say that for any connected Abelian subgroup $H$ of the biholomorphisms of $D$, the moment map is a function $\mu^H : D \rightarrow \fh^*$, where $\fh^*$ is the dual vector space of $\fh$. Every such subgroup $H$ is, up to conjugacy, contained in one of the MASGs listed in Section~\ref{sec:MASG_momentmaps}. For the latter, the Lie algebra is $\R^n$ in all cases. Hence, after some identifications we can consider the moment map of $H$ as a map $\mu^H : D \rightarrow \fh$, where $\fh$ is some subspace of $\R^n$. Then, we consider the symbols $a : D \rightarrow \C$ that can be written as
\[
    a(z) = f(\mu^H(z))
\]
for every $z \in D$ and some function $f$. We call these moment map functions or $\mu$-functions (see Definition~\ref{def:momentmapfunction}). The space of essentially bounded moment map functions for $H$ will be denoted by $L^\infty(D)^{\mu^H}$. In this work we show that these spaces of special symbols give a positive answer to the question formulated above. More precisely, for every connected Abelian subgroup $H$ the $C^*$-algebra generated by Toeplitz operators whose symbols lie in $L^\infty(D)^{\mu^H}$ is commutative on every weighted Bergman space. This is the content of Theorem~\ref{thm:Toeplitz_mu-functions}.

We also prove that when $H$ is a MASG of the biholomorphisms of $D$, we have $L^\infty(D)^{\mu^H} = L^\infty(D)^H$, see Proposition~\ref{prop:H-invariance_momentmaps}. In words, for MASGs the invariant symbols and the moment map functions are exactly the same. As stated after Definition~\ref{def:momentmapfunction} this implies that, in the case of MASGs, the $C^*$-algebras considered in \cite{QVUnitBall1} are precisely our $C^*$-algebras obtained from moment map functions. Hence, our Theorem~\ref{thm:Toeplitz_mu-functionsMASG} is the reformulation of the commutativity results for symbols invariant under MASGs in terms of moment map functions. And so Theorem~\ref{thm:Toeplitz_mu-functions} is an extension to arbitrary connected Abelian subgroups, through our moment map constructions, of the commutativity results from \cite{QVUnitBall1}.

As it was already observed in \cite{BauerVasilevskiQuasiNilpotent2012}, \cite{BauerVasilevskiQuasiHyperbolic2012}, \cite{DG-HM-SN}, \cite{GQVJFA}, \cite{MR-QB-SN}, \cite{QVUnitBall1}, \cite{QVUnitBall2}, \cite{SVNilpotentDuduchava2017},  \cite{VasilevskiQuasiRadial2010}, \cite{VasilevskiParabolicQuasiRadial2010}, it is very important and enlightening to have explicit computations for the symbols and Toeplitz operators under study. Hence, we compute in Section~\ref{sec:MASG_momentmaps} the moment maps for all the MASGs of the biholomorphisms of $D$. This is how we prove that invariant symbols and moment map functions for MASGs are the same as noted above.

In Section~\ref{sec:commToep_momentmaps_Abelian} we give a formula in Proposition~\ref{prop:momentmap_AbelianfromMASG} to compute the moment map of a connected Abelian subgroup $H$ in terms of the moment map of a MASG $G$ that contains it. We also discuss in this section how to describe all the connected Abelian subgroups of a given MASG $G$. Furthermore, we introduce the use of a basis $\beta$ for the Lie algebra $\fh$ of such a connected Abelian subgroup $H$ to describe the moment map functions for $H$. This leads us to a coordinate oriented description of the symbols defined by moment map functions that depend only on the choice of a MASG $G$, where $H$ lies, and a linearly independent subset $\beta \subset \R^n$, a basis for the Lie algebra $\fh$ of $H$. We proceed with such a description for each MASG listed in Section~\ref{sec:MASG_momentmaps} to obtain five families of special symbols discussed in separate subsections: $\beta$-quasi-elliptic, $\beta$-quasi-parabolic, $\beta$-quasi-hyperbolic, $\beta$-nilpotent and $\beta$-quasi-nilpotent. We will refer to these families of symbols collectively as $\beta$-symbols. Each class defines a smaller set of symbols than the corresponding set of invariant symbols. Hence, we obtain commutativity results for the $C^*$-algebras generated by Toeplitz operators with these $\beta$-symbols. The corresponding results are stated in respective subsections of Section~\ref{sec:commToep_momentmaps_Abelian} for all five types.

It turns out that our $\beta$-symbols contain as particular cases all of the special symbols that appear in the literature, to the best of our knowledge, whose Toeplitz operators generate commutative $C^*$-algebras on every weighted Bergman space on the unit ball. This is discussed in the subsections of Section~\ref{sec:commToep_momentmaps_Abelian}. In fact, we prove that, for the quasi-elliptic case, the radial, quasi-radial and separately radial symbols found in \cite{VasilevskiQuasiRadial2010} can be realized as $\beta$-quasi-elliptic symbols for corresponding sets $\beta$. The parabolic quasi-radial symbols found in \cite{VasilevskiParabolicQuasiRadial2010} are $\beta$-quasi-parabolic symbols for some $\beta$. The quasi-nilpotent quasi-radial symbols from \cite{BauerVasilevskiQuasiNilpotent2012} are $\beta$-quasi-nilpotent for a suitable $\beta$. Furthermore, we show in all these cases that there are some $\beta$ for which the corresponding $\beta$-symbols cannot be realized as symbols from these references (see the corresponding subsections). Hence, we have effectively introduced new sets of symbols.

The cases of the $\beta$-quasi-hyperbolic and the $\beta$-nilpotent symbols require some special remarks.

In the previous literature, it has been observed the construction of new symbols from the MASGs by using a sort of quasi-radial procedure. This consists on replacing the dependence of $(|z_1|, \dots, |z_\ell|)$ (with $\ell$ depending on the MASG) by dependence of $(|z_{(k_1)}|, \dots, |z_{(k_m)}|)$, where $k$ is a partition of $\ell$ (see the notation and remarks from Example~\ref{ex:quasi-radial_partitions}). However, this does not work with the nilpotent MASG since there is no toral part for this subgroup. The only new special symbols introduced from the nilpotent case before can be found in \cite{SVNilpotentDuduchava2017} which basically simply eliminates the dependence on some coordinates. Our $\beta$-nilpotent symbols contain this as a particular case (see Example~\ref{ex:beta-nilpotent-noradial}) but provides a much larger variety of new symbols from the nilpotent case.

On the other hand, we have proved that the quasi-hyperbolic symbols are special cases of our $\beta$-quasi-hyperbolic symbols, as it follows from Remark~\ref{rmk:Hmomentfunctions_Gmomentfunctions} and the corresponding definitions. Also, we prove in Subsection~\ref{subsec:beta-quasi-hyperbolic-spectra} how to relate the coordinates used in \cite{BauerVasilevskiQuasiHyperbolic2012} and \cite{QVUnitBall1} to define and treat the quasi-hyperbolic symbols with the coordinates defined by the moment map of the quasi-hyperbolic action. It seems to us that the latter coordinates are somewhat simpler than the former. Hence, it is possible that the use of moment map coordinates can simplify some of the current integral spectral formulas in this case.

Finally, Section~\ref{sec:spectralformulas} provides explicit formulas for the spectra of the Toeplitz operators with $\beta$-symbols. We would like to remark that such formulas are particularly simple in comparison with those from previous works. This is so even for the case of symbols that are invariant with respect to MASGs. Somehow this signals the fact that coordinates oriented by moment maps seem to be most natural to understand this sort of commutative $C^*$-algebras generated by Toeplitz operators.

We would like to mention some previously related results. These can be found in \cite{DG-HM-SN} and \cite{MR-QB-SN}, where the authors considered symplectic geometric constructions on the  poly-disk $\Delta^{n}$ of dimension $n$ where $\Delta$ denotes the unit disk and the $n$-dimensional projective space $\mathbb{P}^{n}(\C)$, respectively. Such constructions lead to commutative $C^*$ and Banach algebras generated by Toeplitz operators on the Bergman spaces of these manifolds.

\section{Unit ball analysis}\label{sec:unitball_analysis}
Let us denote by $\dif v$ the Lebesgue measure on the $n$-dimensional unit ball $\B^n$, and for every $\lambda > -1$ let us consider the weighted measure
\[
    \dif v_\lambda(z) = c_\lambda (1-|z|^2)^\lambda \dif v(z)
\]
where $c_\lambda$ is the constant given by
\[
    c_\lambda = \frac{\Gamma(n+1+\lambda)}{\pi^n \Gamma(\lambda+1)}
\]
which is chosen so that $\dif v_\lambda$ is a probability measure. As it is customary, this leads to the weighted Bergman space $\cA^2_\lambda(\B^n)$, defined for $\lambda > -1$, which consists of the holomorphic functions on $\B^n$ that belong to $L^2(\B^n,v_\lambda)$. This is well known to be a reproducing kernel Hilbert space with Bergman kernel given by
\[
    K_{\B^n,\lambda}(z,w) = \frac{1}{(1-z\cdot \overline{w})^{\lambda+n+1}}.
\]
The correspoding orthogonal projection $B_{\B^n,\lambda} : L^2(\B^n,v_\lambda) \rightarrow \cA^2_\lambda(\B^n)$ is thus given by
\[
    B_{\B^n,\lambda}(f)(z)
    = \int_{\B^n} \frac{f(w) \dif v_\lambda(w)}{(1-z\cdot\overline{w})^{\lambda+n+1}}
    = c_\lambda \int_{\B^n} \frac{f(w) (1-|w|^2)^\lambda \dif v(w)}{(1-z\cdot\overline{w})^{\lambda+n+1}}
\]
for every $f \in L^2(\B^n,v_\lambda)$ and $z \in \B^n$. In what follows the weightless case, corresponding to $\lambda = 0$, will be denoted by dropping the subindex $\lambda$. In particular, the weightless Bergman kernel is denoted by $K_{\B^n} = K_{\B^n,0}$.

For every $a \in L^\infty(\B^n)$ we consider the Toeplitz operator with symbol $a$ given by
\begin{align*}
    T_a = T^{(\lambda)}_a : \cA^2_\lambda(\B^n) &\longrightarrow \cA^2_\lambda(\B^n) \\
    T_a(f) &= B_{\B^n,\lambda}(af).
\end{align*}
In other words, we have
\[
    T_a(f)(z) = c_\lambda \int_{\B^n} \frac{a(w) f(w) (1-|w|^2)^\lambda \dif v(w)}{(1-z\cdot\overline{w})^{\lambda+n+1}}
\]
for every $f \in \cA^2_\lambda(\B^n)$ and $z \in \B^n$.

On the other hand, we let $D_n$ denote the unbounded Siegel domain realization of $\B^n$ given by
\[
    D_n = \{ z = (z',z_n) \in \C^{n-1}\times\C \mid \im(z_n) > |z'|^2 \}.
\]
In this work, for $z \in \C^n$ we will denote by $z' \in \C^{n-1}$ the first $n-1$ components of $z$ so that we have $z = (z',z_n)$.

The well known Cayley transformation $\varphi$ and its inverse $\psi$ relating $\B^n$ and $D_n$ are given by
\begin{align*}
    \varphi : \B^n &\longrightarrow D_n  &
            \psi : D_n &\longrightarrow \B^n \\
    z &\longmapsto \frac{i}{1+z_n}\left(z',1-z_n\right), &
            z &\longmapsto \frac{1}{1-iz_n} \left(-2iz', 1+iz_n \right).
\end{align*}
On the Siegel domain $D_n$ and for every $\lambda > -1$ we consider the weighted measure
\[
    \dif \widehat{v}_\lambda (z) =
        \frac{c_\lambda}{4} (\im(z_n) - |z'|^2)^\lambda \dif v(z).
\]
Then, the weighted Bergman space on $D_n$ with weight $\lambda > -1$ is denoted by $\cA^2_\lambda(D_n)$ and consists of the holomorphic functions on $D_n$ that belong to $L^2(D_n,\widehat{v}_\lambda)$. There is a natural correspondence between the Bergman spaces $\cA^2_\lambda(D_n)$ and $\cA^2_\lambda(\B^n)$ given by $\psi$. More precisely, the linear map
\begin{align*}
    U_\lambda : L^2(\B^n,v_\lambda) &\longrightarrow
        L^2(D_n,\widehat{v}_\lambda) \\
    (U_\lambda(f))(z) &= \Big(\frac{2}{1-iz_n}\Big)^{\lambda+n+1}
            f(\psi(z)),
\end{align*}
for all $z \in D_n$, is a well defined unitary operator such that
\[
    U_\lambda(\cA^2(\B^n)) = \cA^2_\lambda(D_n).
\]
This and a straightforward computation shows that $\cA^2_\lambda(D_n)$ is a reproducing kernel Hilbert space with kernel given by
\[
    K_{D_n,\lambda}(z,w) = \frac{1}{\Big(
        \frac{z_n - \overline{w}_n}{2i} - z'\cdot \overline{w}'
        \Big)^{\lambda+n+1}},
\]
for every $\lambda > -1$. As before, the weightless kernel is denoted by dropping the subindex $\lambda$. Hence, we denote $K_{D_n} = K_{D_n,0}$.

By the same token, for every $a \in L^\infty(D_n)$ we define the corresponding Toeplitz operator on $\cA^2_\lambda(D_n)$ with symbol $a$ by
\begin{align*}
    T_a = T^{(\lambda)}_a : \cA^2_\lambda(D_n) &\longrightarrow \cA^2_\lambda(D_n) \\
    T_a(f) &= B_{D_n,\lambda}(af),
\end{align*}
which is thus explicitly given by
\[
    T_a(f)(z) = \frac{c_\lambda}{4} \int_{D_n} \frac{a(w) f(w) (\im(w_n)-|w'|^2)^\lambda \dif v(w)}{\Big(
        \frac{z_n - \overline{w}_n}{2i} - z'\cdot \overline{w}'
        \Big)^{\lambda+n+1}}
\]
for every $f \in \cA^2_\lambda(D_n)$ and $z \in D_n$.

\section{Unit ball geometry}\label{sec:unitball_geometry}
To define the Bergman metric of the unit ball $\B^n$ we consider the matrix coefficient functions
\[
    (g_{\B^n})_{jk}(z) = \frac{1}{n+1}
        \frac{\partial^2}{\partial z_j \partial\overline{z}_k}
            \log K_{\B^n}(z,z)
        = \frac{(1-|z|^2)\delta_{jk} + \overline{z}_j z_k}{(1-|z|^2)^2}
\]
where $j,k = 1, \dots, n$. This yields the Bergman metric $g_{\B^n}$ on $\B^n$, which is explicitly given as follows
\begin{align*}
    (g_{\B^n})_z &=\sum_{j,k=1}^n \frac{(1-|z|^2)\delta_{jk} +
                \overline{z}_j z_k}{(1-|z|^2)^2}
                    \dif z_j \otimes \dif \overline{z}_k  \\
        &= \frac{1}{(1-|z|^2)^2} \bigg(
            (1-|z|^2)\sum_{j=1}^n \dif z_j \otimes \dif \overline{z}_j +
            \sum_{j,k=1}^n \overline{z}_j z_k \dif z_j \otimes \dif \overline{z}_k\bigg),
\end{align*}
where $z \in \B^n$. This is a very well known example of a K\"ahler metric whose associated K\"ahler form is given by
\begin{align*}
    (\omega_{\B^n})_z &= i\sum_{j,k=1}^n \frac{(1-|z|^2)\delta_{jk} +
                \overline{z}_j z_k}{(1-|z|^2)^2}
                    \dif z_j \wedge \dif \overline{z}_k \\
        &= \frac{i}{(1-|z|^2)^2} \bigg(
            (1-|z|^2)\sum_{j=1}^n \dif z_j \wedge \dif \overline{z}_j +
                \sum_{j,k=1}^n \overline{z}_j z_k \dif z_j \wedge \dif \overline{z}_k\bigg),
\end{align*}
where $z \in \B^n$.

For every smooth function $f$ on $\B^n$, the Hamiltonian vector field associated to $f$ is the smooth vector field $X_f$ that satisfies
\[
    \dif f(X) = \omega_{\B^n}(X_f, X)
\]
for every smooth vector field $X$. In particular, we have the well known formula
\begin{align}
    X_f(z) &= i \sum_{j,k=1}^n (g_{\B^n})^{jk}(z)
            \bigg(
            \frac{\partial f}{\partial z_k} \frac{\partial}{\partial \overline{z}_j} -
            \frac{\partial f}{\partial \overline{z}_j} \frac{\partial}{\partial z_k}
            \bigg) \notag     \\
     &= i(1-|z|^2) \sum_{j,k=1}^n (\delta_{jk} - \overline{z}_j z_k)
            \bigg(
            \frac{\partial f}{\partial z_k} \frac{\partial}{\partial \overline{z}_j} -
            \frac{\partial f}{\partial \overline{z}_j} \frac{\partial}{\partial z_k}
            \bigg) \label{eq:Hamiltonian_Bn},
\end{align}
at every $z \in \B^n$, where $((g_{\B^n})^{jk}(z))_{j,k}$ denotes the inverse matrix of $((g_{\B^n})_{jk}(z))_{j,k}$.

Correspondingly, the Bergman metric of the Siegel domain $D_n$ is defined by the following matrix coefficients
\[
    (g_{D_n})_{jk}(z) = \frac{1}{n+1}
        \frac{\partial^2}{\partial z_j \partial\overline{z}_k}
            \log K_{D_n}(z,z),
\]
where $j,k=1, \dots, n$. In this case, the K\"ahler metric obtained is given by
\begin{multline*}
    (g_{D_n})_z = \frac{1}{(\im(z_n)-|z'|^2)^2}
        \bigg( (\im(z_n)-|z'|^2) \sum_{j=1}^{n-1}
                \dif z_j \otimes \dif \overline{z}_j
            + \sum_{j,k=1}^{n-1} \overline{z}_j z_k
                \dif z_j \otimes \dif \overline{z}_k \\
          + \frac{1}{2i}
            \sum_{j=1}^{n-1} (\overline{z}_j \dif z_j\otimes \dif\overline{z}_n
                - z_j \dif z_n \otimes \dif \overline{z}_j)
          + \frac{1}{4} \dif z_n \otimes \dif \overline{z}_n
            \bigg),
\end{multline*}
where $z \in D_n$. It follows that the K\"ahler form of $D_n$ is given by
\begin{multline*}
    (\omega_{D_n})_z = \frac{i}{(\im(z_n)-|z'|^2)^2}
        \bigg( (\im(z_n)-|z'|^2) \sum_{j=1}^{n-1}
                \dif z_j \wedge \dif \overline{z}_j
            + \sum_{j,k=1}^{n-1} \overline{z}_j z_k
                \dif z_j \wedge \dif \overline{z}_k  \\
          + \frac{1}{2i}
            \sum_{j=1}^{n-1} (\overline{z}_j \dif z_j\wedge \dif\overline{z}_n
                - z_j \dif z_n \wedge \dif \overline{z}_j)
          + \frac{1}{4} \dif z_n \wedge \dif \overline{z}_n
            \bigg),
\end{multline*}
where $z \in D_n$.

A simple computation shows that
\[
    K_{D_n}(z,z) = 4 |J_\C \psi(z)|^2 K_{\B^n}(\psi(z),\psi(z))
\]
for every $z \in D_n$, and this is easily seen to imply that $\psi$, and so $\phi$ as well, is an isometry for the Bergman metrics on $\B^n$ and $D_n$.

\begin{lemma}\label{lem:Hamiltonian_Dn}
    For every smooth function $f$ on $D_n$, the Hamiltonian vector field associated to $f$ is given by
    \begin{align*}
        X_f&(z) =  \\
        =&\; i(\im(z_n) - |z'|^2) \times \notag  \\
        &\times
            \Bigg[
                \sum_{k=1}^{n-1}
                    \bigg(
                    \frac{\partial f}{\partial z_k}
                    +2i \overline{z}_k \frac{\partial f}{\partial z_n}
                    \bigg) \frac{\partial}{\partial \overline{z}_k} -
                \sum_{k=1}^{n-1}
                    \bigg(
                    \frac{\partial f}{\partial \overline{z}_k}
                        -2i z_k \frac{\partial f}{\partial \overline{z}_n}
                    \bigg) \frac{\partial}{\partial z_k}   \\
        &+ \bigg(
                4\im(z_n) \frac{\partial f}{\partial z_n}
                -2i \sum_{k=1}^{n-1} z_k \frac{\partial f}{\partial z_k}
            \bigg) \frac{\partial}{\partial \overline{z}_n}
        - \bigg(
                4\im(z_n) \frac{\partial f}{\partial \overline{z}_n}
                +2i \sum_{k=1}^{n-1} \overline{z}_k
                            \frac{\partial f}{\partial \overline{z}_k}
                \bigg) \frac{\partial}{\partial z_n} \Bigg],
    \end{align*}
    for every $z \in D_n$.
\end{lemma}
\begin{proof}
    Similar to the case of the unit ball $\B^n$, as noted above in Equation~\eqref{eq:Hamiltonian_Bn}, the Hamiltonian vector field associated to $f$ is given by the following formula
    \[
        X_f(z)
            = i\sum_{j,k=1}^n (g_{D_n})^{jk}(z)
            \bigg(
            \frac{\partial f}{\partial z_k}
                \frac{\partial}{\partial\overline{z}_j}
            - \frac{\partial f}{\partial \overline{z}_j}
                \frac{\partial}{\partial z_k}
                \bigg),
    \]
    at every $z \in D_n$. A straightforward computation shows that the inverse matrix of $((g_{D_n})_{jk}(z))_{j,k=1}^n$ is given by
    \[
        ((g_{D_n})^{jk}(z))_{j,k=1}^n =
        (\im(z_n) - |z'|^2)
            \begin{pmatrix}
                I_{n-1} & 2i \overline{z}'^\top \\
                -2i z' & 4\im(z_n)
            \end{pmatrix},
    \]
    for every $z \in D_n$. Hence, the Hamiltonian vector field associated to $a$ is given by
    \begin{align*}
        X_f(z) &= \\
        =&\; i(\im(z_n) - |z'|^2) \Bigg[
            \sum_{k=1}^{n-1}
            \bigg(
            \frac{\partial f}{\partial z_k}
                \frac{\partial}{\partial\overline{z}_k}
            - \frac{\partial f}{\partial \overline{z}_k}
                \frac{\partial}{\partial z_k}
            \bigg)  \\
        &+ 2i \sum_{k=1}^{n-1} \overline{z}_k
            \bigg(
            \frac{\partial f}{\partial z_n}
                \frac{\partial}{\partial\overline{z}_k}
            - \frac{\partial f}{\partial \overline{z}_k}
                \frac{\partial}{\partial z_n}
            \bigg)
            -2i \sum_{k=1}^{n-1} z_k
            \bigg(
            \frac{\partial f}{\partial z_k}
                \frac{\partial}{\partial\overline{z}_n}
            - \frac{\partial f}{\partial \overline{z}_n}
                \frac{\partial}{\partial z_k}
            \bigg) \\
        &+ 4\im(z_n)
            \bigg(
            \frac{\partial f}{\partial z_n}
                \frac{\partial}{\partial\overline{z}_n}
            - \frac{\partial f}{\partial \overline{z}_n}
                \frac{\partial}{\partial z_n}
            \bigg)
        \Bigg],
    \end{align*}
    and this yields the result after collecting terms.
\end{proof}

Let $H$ be a connected Lie group with Lie algebra $\fh$ and assume that $H$ acts smoothly on either $\B^n$ or $D_n$. Then, for every $X \in \fh$ there is a smooth vector field given by
\[
    X^\sharp_z = \frac{\dif}{\dif s}\sVert[2]_{s=0} \exp(sX) z
\]
for every $z$ in the corresponding domain. If we consider $X^\sharp$ as a complex-valued function with component functions given by $X^\sharp = (f_1, \dots, f_n)$, then it is easy to see that the corresponding expression as a complex vector field is the following
\[
    X^\sharp = \sum_{j=1}^n
        \Big( f_j \frac{\partial}{\partial z_j} +
            \overline{f}_j \frac{\partial}{\partial \overline{z}_j}
            \Big).
\]

\begin{definition}\label{def:moment_map}
Let $D$ denote either $\B^n$ or $D_n$, and let $H$ be a connected Lie group acting smoothly on $D$ and preserving its K\"ahler form $\omega$. Let us denote by $\fh^*$ the real dual space of $\fh$. A moment map for the $H$-action on $D$ is a smooth function
\[
    \mu : D \rightarrow \fh^*
\]
that satisfies the following properties.
\begin{enumerate}
    \item For every $X \in \fh$, the smooth function $\mu_X : D \rightarrow \R$ defined by
        \[
            \mu_X(z) = \left<\mu(z),X\right>
        \]
        has Hamiltonian vector field given by $X^\sharp$. In other words, we have $X^\sharp = X_{\mu_X}$, for every $X \in \fh$.
    \item For every $h \in H$ we have $\mu \circ h = \Ad(h)^*\circ \mu$, where $\Ad$ is the adjoint representation of $H$ and $\Ad(h)^*$ denotes the transpose transformation of $\Ad(h)$.
\end{enumerate}
\end{definition}

If $H$ is an Abelian Lie group, then its adjoint representation satisfies $\Ad(h) = I_{\fh}$, the identity map on $\fh$, for every $h \in H$. Hence, for Abelian Lie groups the second condition in Definition~\ref{def:moment_map} is equivalent to the $H$-invariance of $\mu$, i.e.~it is the same as requiring that
\[
    \mu \circ h = \mu,
\]
for all $h \in H$.

\section{Maximal Abelian subgroups and their moment maps} \label{sec:MASG_momentmaps}
We will consider the pseudo-unitary Lie group with signature $(n,1)$ defined by
\[
    \SU(n,1) = \{ M \in \GL(n+1,\C) \mid M I_{n,1} \overline{M}^\top =
            I_{n,1},\; \det(M) = 1 \},
\]
where we have
\[
    \begin{pmatrix}
      I_n & 0 \\
      0 & -1
    \end{pmatrix}.
\]
We recall that $\SU(n,1)$ realizes the group of biholomorphisms of $\B^n$ through its action by linear fractional transformations
\[
    M\cdot z = \frac{A z + b}{c\cdot z +d}
\]
where the elements of $\B^n$ are taken as columns and we use the block decomposition
\[
    M =
    \begin{pmatrix}
      A & b \\
      c & d
    \end{pmatrix} \in \SU(n,1),
\]
such that $A$ has size $n\times n$ and $d \in \C$.

For the Siegel realization $D_n$ we consider the matrix
\[
    L_n =
    \begin{pmatrix}
      2 I_{n-1} & 0 & 0 \\
      0 & 0 & -i \\
      0 & i & 0
    \end{pmatrix},
\]
that yields the special pseudo-unitary Lie group
\[
    \SU(\C^{n+1}, L_n) =
        \{ M \in \GL(n+1,\C) \mid M L_n \overline{M}^\top = L_n,\;
            \det(M) = 1 \}.
\]
This realizes the group of biholomorphisms of $D_n$ through the action by linear fractional transformations
\[
    M \cdot z =
    \frac{(A z' + \alpha z_n + \beta, \gamma\cdot z' + az_n + b)}{\delta\cdot z' + cz_n + d}
\]
where the elements of $D_n$ are taken as columns and we use the block decomposition
\[
    M =
    \begin{pmatrix}
      A & \alpha & \beta \\
      \gamma & a & b \\
      \delta & c & d
    \end{pmatrix} \in \SU(\C^{n+1},L_n),
\]
where $A$ has size $(n-1)\times(n-1)$ and $a, d \in \C$.

It is well known that the group of biholomorphisms for either $\B^n$ or $D_n$ has exactly $n+2$ conjugacy classes of (connected) maximal Abelian subgroups, MASG for short. The following is a complete list of representatives of these classes.

\textbf{Quasi-elliptic}: This is denoted by $E(n)$ and it is given by the $\T^n$-action on $\B^n$
\[
    t \cdot z = (t_1 z_1, \dots, t_n z_n),
\]
where $t \in \T^n$ and $z \in \B^n$.

\textbf{Quasi-parabolic}: This is denoted by $P(n)$ and it is given by the $\T^{n-1}\times\R$-action on $D_n$
\[
    (t',h) \cdot z = (t' z', z_n +h)
\]
where $t' \in \T^{n-1}$, $h \in \R$ and $z \in D_n$.

\textbf{Quasi-hyperbolic}: This is denoted by $H(n)$ and it is given by the $\T^{n-1}\times\R_+$-action on $D_n$
\[
    (t',r) \cdot z = (r^\frac{1}{2} t' z', r z_n)
\]
where $t' \in \T^{n-1}$, $r \in \R_+$ and $z \in D_n$.

\textbf{Nilpotent}: This is denoted by $N(n)$ and it is given by the $\R^n$-action on $D_n$
\[
    (b,h) \cdot z = (z'+b, z_n + h + 2iz'\cdot b + i |b|^2)
\]
where $b \in \R^{n-1}$, $h \in \R$ and $z \in D_n$.

\textbf{Quasi-nilpotent}: For every integer $k = 1, \dots, n-2$ we have a group action denoted by $N(n,k)$ and that is given by
\[
    (t, b, h) \cdot z = (t z_{(1)}, z_{(2)} +b, z_n + h + 2iz_{(2)}\cdot b + i |b|^2)
\]
where $t \in \T^k$, $b \in \R^{n-k-1}$, $h \in \R$ and $z \in D_n$ is decomposed as $z = (z_{(1)}, z_{(2)}, z_n) \in \C^k\times\C^{n-k-1}\times\C$.

We observe that the notation in the last case corresponds to the decomposition in components associated to the partition of $n$ given by $(k,n-k-1,1)$. In particular, the last component could also be written as $z_n = z_{(3)}$. In the rest of this work, we will use the notation $z = (z_{(1)}, z_{(2)}, z_n)$ for this particular decomposition. Although the dependence on $k$ is not made explicit in such decomposition, it will always be used in connection to the action $N(n,k)$ and so there should not be any confusion.

We observe that the Lie algebra of all the groups involved in these actions is $\R^n$ in all cases, which is naturally identified with its dual $(\R^n)^*$ by the usual inner product. The description of the exponential maps of these groups can be reduced to three cases. The simplest one corresponds to the additive group $\R$, whose exponential map is just the identity. The exponential map for the multiplicative group $\R_+$ is the map $\R \rightarrow \R_+$ given by $s \mapsto e^s$. Finally, the exponential map for the circle group $\T$ is the map $\R \rightarrow \T$ given by $s \mapsto e^{is}$.

\subsection{Moment map for $E(n)$}
Let $X \in \R^n$ be given with components $X = (s_1, \dots, s_n)$. Then, for the group $\T^n$ we have the exponential
\[
    \exp(s X) = (e^{is s_1}, \dots, e^{is s_n}),
\]
for every $s \in \R$. Its action on a given $z \in \B^n$ is
\[
    \exp(s X)\cdot z = (e^{is s_1} z_1, \dots, e^{is s_n} z_n),
\]
for every $s \in \R$. And so the vector field $X^\sharp$ as a function $\B^n \rightarrow \C^n$ is given by
\[
    X^\sharp_z = (is_1 z_1, \dots, is_n z_n),
\]
for all $z \in \B^n$. Hence, we have the following expression of $X^\sharp$ as complex vector field
\[
    X^\sharp_z = i\sum_{j=1}^n \Big(s_j z_j \frac{\partial}{\partial z_j}
    - s_j \overline{z}_j \frac{\partial}{\partial \overline{z}_j}\Big)
\]
for all $z \in \B^n$.

\begin{proposition}\label{prop:momentmap_E(n)}
    The function given by
    \begin{align*}
        \mu : \B^n &\longrightarrow \R^n \\
        \mu(z) &= -\frac{1}{1-|z|^2}(|z_1|^2, \dots, |z_n|^2)
    \end{align*}
    is a moment map for the $E(n)$ action on $\B^n$.
\end{proposition}
\begin{proof}
    For every $X \in \R^n$, with components $X = (s_1, \dots, s_n)$ as before, the function $\mu_X : \B^n \rightarrow \R$ given by
    \[
        \mu_X(z) = \left<\mu(z), X\right>
    \]
    for every $z \in \B^n$, satisfies
    \[
        \mu_X(z) = -\frac{1}{1-|z|^2}\sum_{j=1}^n s_j|z_j|^2
    \]
    for every $z \in \B^n$. Here, we have used the natural identification between $\R^n$ and its dual. Then, it is easy to see that
    \begin{align*}
        \frac{\partial \mu_X}{\partial z_j}(z) &=
            -\frac{s_j \overline{z}_j}{1-|z|^2}
            -\frac{\overline{z}_j}{(1-|z|^2)^2}\sum_{k=1}^n s_k |z_k|^2, \\
        \frac{\partial \mu_X}{\partial \overline{z}_j}(z) &=
            -\frac{s_j z_j}{1-|z|^2}
            -\frac{z_j}{(1-|z|^2)^2}\sum_{k=1}^n s_k |z_k|^2,
    \end{align*}
    for all $j = 1,\dots, n$. Hence, a straightforward replacement into Equation~\eqref{eq:Hamiltonian_Bn} shows that the Hamiltonian vector field of $\mu_X$ is $X^\sharp$. Since the function $\mu$ is clearly $\T^n$-invariant this proves the result.
\end{proof}

\subsection{Moment map for $P(n)$}
Let $X \in \R^n$ be given with components $X = (s_1, \dots, s_n)$. In this case, the exponential for the group $\T^{n-1}\times\R$ yields
\[
    \exp(s X) = (e^{is s_1}, \dots, e^{is s_{n-1}}, s s_n),
\]
for every $s \in \R$. Its action on a given $z \in D_n$ is now
\[
    \exp(s X)\cdot z
    = (e^{is s_1} z_1, \dots, e^{is s_{n-1}} z_{n-1}, z_n + s s_n),
\]
for every $s \in \R$. Hence, the vector field $X^\sharp$ as a function $D_n \rightarrow \C^n$ is given by
\[
    X^\sharp_z = (is_1 z_1, \dots, is_{n-1} z_{n-1}, s_n),
\]
for all $z \in D_n$. Hence, $X^\sharp$ as a complex vector field is given by the following expression
\[
    X^\sharp_z = i\sum_{j=1}^{n-1}\Big(s_j z_j \frac{\partial}{\partial z_j}
        - s_j \overline{z}_j \frac{\partial}{\partial \overline{z}_j} \Big)
        + s_n \Big(
        \frac{\partial}{\partial z_n}
            + \frac{\partial}{\partial \overline{z}_n}
        \Big)
\]
for all $z \in D_n$.

\begin{proposition}\label{prop:momentmap_P(n)}
    The function given by
    \begin{align*}
        \mu : D_n &\longrightarrow \R^n \\
        \mu(z) &= -\frac{1}{2(\im(z_n)-|z'|^2)}
                (2|z_1|^2, \dots, 2|z_{n-1}|^2, 1),
    \end{align*}
    is a moment map for the action $P(n)$ on $D_n$.
\end{proposition}
\begin{proof}
    For every $X \in \R^n$ whose components are $X = (s_1, \dots, s_n)$, we have $\mu_X : D_n \rightarrow \R$ given by
    \begin{align*}
        \mu_X(z) &= \left<\mu(z), X\right> \\
            &= -\frac{1}{2(\im(z_n)-|z'|^2)}
                \bigg(
                    \sum_{j=1}^{n-1} 2 s_j |z_j|^2 + s_n
                \bigg).
    \end{align*}
    Since $\mu$ is clearly $\T^{n-1}\times\R$-invariant, it is enough to show that $X^\sharp$ is the Hamiltonian vector field of $\mu_X$, for all $X \in \R^n$. In other words that
    \[
        X^\sharp = X_{\mu_X}
    \]
    for all $X \in \R^n$. Since both sides of the last expression are linear in $X$, it is enough to check this on the canonical basis of $\R^n$ which we will denote by $X_1, \dots, X_n$. In other words, if we denote $\mu_j = \mu_{X_j}$ ($j=1, \dots, n$), then it is enough to show that $X_j^\sharp$ is the Hamiltonian vector field of $\mu_j$ for all $j = 1, \dots, n$. These functions have the following expressions
    \[
    \mu_j(z) =
    \begin{cases}
        \displaystyle
        -\frac{|z_j|^2}{\im(z_n) - |z'|^2}, & \mbox{if } j=1,\dots,n-1 \\
        \displaystyle
        -\frac{1}{2(\im(z_n) - |z'|^2)}, & \mbox{if } j=n.
    \end{cases}
    \]
    for all $z \in D_n$. And from the previous computations we have
    \[
    X_j^\sharp(z) =
    \begin{cases}
        \displaystyle
        iz_j \frac{\partial}{\partial z_j}
            - i\overline{z}_j \frac{\partial}{\partial \overline{z}_j}, & \mbox{if } j=1,\dots,n-1 \\
        \displaystyle
        \frac{\partial}{\partial z_n}
            + \frac{\partial}{\partial \overline{z}_n}, & \mbox{if } j=n.
    \end{cases}
    \]
    So, with this information, we need to check that $X_{\mu_j} = X_j^\sharp$ for all $j =1, \dots, n$.

    First we observe that for $j = 1, \dots, n-1$ we have
    \begin{align*}
    \frac{\partial \mu_j}{\partial z_k}(z) &=
    \begin{cases}
        \displaystyle
        -\frac{\delta_{jk}\overline{z}_k}{\im(z_n) - |z'|^2}
        -\frac{|z_j|^2\overline{z}_k}{(\im(z_n) - |z'|^2)^2},
            & \mbox{if } k=1,\dots,n-1 \\
        \displaystyle
        \frac{|z_j|^2}{2i(\im(z_n) - |z'|^2)^2}, & \mbox{if } k=n,
    \end{cases} \\
    \frac{\partial \mu_j}{\partial \overline{z}_k}(z) &=
    \begin{cases}
        \displaystyle
        -\frac{\delta_{jk} z_k}{\im(z_n) - |z'|^2}
        -\frac{|z_j|^2 z_k}{(\im(z_n) - |z'|^2)^2},
            & \mbox{if } k=1,\dots,n-1  \\
        \displaystyle
        -\frac{|z_j|^2}{2i(\im(z_n) - |z'|^2)^2}, & \mbox{if } k=n.
    \end{cases}
    \end{align*}
    On the other hand, for the function $\mu_n$ we have
    \begin{align*}
    \frac{\partial \mu_n}{\partial z_k}(z) &=
    \begin{cases}
        \displaystyle
        -\frac{\overline{z}_k}{2(\im(z_n) - |z'|^2)^2},
            & \mbox{if } k=1,\dots,n-1 \\
        \displaystyle
        \frac{1}{4i(\im(z_n) - |z'|^2)^2}, & \mbox{if } k=n,
    \end{cases} \\
    \frac{\partial \mu_n}{\partial \overline{z}_k}(z) &=
    \begin{cases}
        \displaystyle
        -\frac{z_k}{2(\im(z_n) - |z'|^2)^2},
            & \mbox{if } k=1,\dots,n-1 \\
        \displaystyle
        -\frac{1}{4i(\im(z_n) - |z'|^2)^2}, & \mbox{if } k=n.
    \end{cases}
    \end{align*}
    After substituting these expressions into the formula from Lemma~\ref{lem:Hamiltonian_Dn} it is easy to conclude the required identities.
\end{proof}

\subsection{Moment map for $H(n)$}
Let $X \in \R^n$ with components $X = (s_1, \dots, s_n)$. Then, for the group $\T^{n-1}\times\R_+$ we have the exponential
\[
    \exp(sX) = (e^{is s_1}, \dots, e^{is s_{n-1}}, e^{s s_n}),
\]
for every $s \in \R$. The action of this element on $z \in D_n$ is given by
\[
    \exp(sX)\cdot z =
        (e^{s s_n/2}e^{is s_1}z_1, \dots, e^{s s_n/2}e^{is s_{n-1}}z_{n-1},
        e^{s s_n}z_n),
\]
for every $s \in \R$. It follows that we have
\[
    X^\sharp_z = \bigg(
    \Big(\frac{s_n}{2} + is_1\Big) z_1, \dots,
    \Big(\frac{s_n}{2} + is_{n-1}\Big)z_{n-1}, s_n z_n
    \bigg)
\]
for every $z \in D_n$, as a function $D_n \rightarrow \C$. Hence, the complex vector field corresponding to $X$ is given by
\[
    X^\sharp_z =
        \sum_{j=1}^{n-1}\bigg(
            \Big(\frac{s_n}{2} + i s_j\Big)z_j \frac{\partial}{\partial z_j}
            +
            \Big(\frac{s_n}{2} - i s_j\Big)\overline{z}_j
            \frac{\partial}{\partial \overline{z}_j}
            \bigg)
        + s_n \bigg(
        z_n \frac{\partial}{\partial z_n}
            + \overline{z}_n \frac{\partial}{\partial \overline{z}_n}
        \bigg),
\]
for all $z \in D_n$.

\begin{proposition}\label{prop:momentmap_H(n)}
    The function given by
    \begin{align*}
        \mu : D_n &\longrightarrow \R^n \\
        \mu(z) &= -\frac{1}{2(\im(z_n) - |z'|^2)}
            (2|z_1|^2, \dots, 2|z_{n-1}|^2, \re(z_n))
    \end{align*}
    is a moment map for the action $H(n)$ on $D_n$.
\end{proposition}
\begin{proof}
    Note that $\mu$ is clearly $\T^{n-1}\times\R_+$-invariant.

    As before, for $X = (s_1, \dots, s_n) \in \R^n$ we have the function $\mu_X : D_n \rightarrow \R$
    \begin{align*}
        \mu_X(z) &= \left<\mu(z), X\right> \\
         &= -\frac{1}{2(\im(z_n) - |z'|^2)}
            \bigg(\sum_{j=1}^{n-1}
                2 s_j |z_j|^2 + s_n \re(z_n).
            \bigg).
    \end{align*}

    As in the proof of Proposition~\ref{prop:momentmap_P(n)} it is enough to show that for the canonical basis $X_1, \dots, X_n$ of $\R^n$ we have
    \[
        X_j^\sharp = X_{\mu_j}
    \]
    where $\mu_j = \mu_{X_j}$, for all $j = 1, \dots, n$. In this case, we have the functions
    \[
    \mu_j(z) =
    \begin{cases}
        \displaystyle
        -\frac{|z_j|^2}{\im(z_n) - |z'|^2}, & \mbox{if } j=1,\dots,n \\
        \displaystyle
        -\frac{\re(z_n)}{2(\im(z_n) - |z'|^2)}, & \mbox{if } j=n,
    \end{cases}
    \]
    as well as the vector fields
    \[
    X_j^\sharp(z) =
    \begin{cases}
        \displaystyle
        iz_j \frac{\partial}{\partial z_j}
            - i\overline{z}_j \frac{\partial}{\partial \overline{z}_j}, & \mbox{if } j=1,\dots,n-1 \\
        \displaystyle
        \frac{1}{2} \sum_{j=1}^{n-1}
        \Big(z_j \frac{\partial}{\partial z_j}
            + \overline{z}_j \frac{\partial}{\partial \overline{z}_j}
        \Big) +
        \Big(z_n \frac{\partial}{\partial z_n}
            + \overline{z}_n \frac{\partial}{\partial \overline{z}_n}
        \Big), & \mbox{if } j=n.
    \end{cases}
    \]
    We observe that the proof of Proposition~\ref{prop:momentmap_P(n)} has already established that $X_j^\sharp = X_{\mu_j}$ for $j = 1, \dots, n-1$. Hence, it remains to prove that $X_n^\sharp = X_{\mu_n}$.

    For the function $\mu_n$ we have
    \begin{align*}
    \frac{\partial \mu_n}{\partial z_k}(z) &=
    \begin{cases}
        \displaystyle
        -\frac{\re(z_n) \overline{z}_k}{2(\im(z_n) - |z'|^2)^2},
            & \mbox{if } k=1,\dots,n-1 \\
        \displaystyle
        -\frac{1}{4(\im(z_n) - |z'|^2)}
        +\frac{\re(z_n)}{4i(\im(z_n) - |z'|^2)^2}, & \mbox{if } k=n,
    \end{cases} \\
    \frac{\partial \mu_n}{\partial \overline{z}_k}(z) &=
    \begin{cases}
        \displaystyle
        -\frac{\re(z_n) z_k}{2(\im(z_n) - |z'|^2)^2},
            & \mbox{if } k=1,\dots,n-1 \\
        \displaystyle
        -\frac{1}{4(\im(z_n) - |z'|^2)}
        -\frac{\re(z_n)}{4i(\im(z_n) - |z'|^2)^2}, & \mbox k=n.
    \end{cases}
    \end{align*}
    Using this expressions we can apply the formula from Lemma~\ref{lem:Hamiltonian_Dn} to complete the proof.
\end{proof}

\subsection{Moment map for $N(n)$}
Let $X = (s_1, \dots, s_n) \in \R^n$ be given. In this case the exponential map in $\R^n$ gives
\[
    \exp(s X) = (s s_1, \dots, s s_n),
\]
for all $s \in \R$. Its action on a given $z \in D_n$ has the expression
\[
    \exp(s X)\cdot z =
    \bigg(z_1 + s s_1, \dots, z_{n-1} + s s_{n-1},
        z_n + s s_n + 2i \sum_{j=1}^{n-1} s s_j z_j
            + i\sum_{j=1}^{n-1} s^2 s_j^2 \bigg).
\]
And from this it follows that, as a function $D_n \rightarrow \C^n$, we have
\[
    X_z^\sharp =
        \bigg(
            s_1, \dots, s_{n-1},
            s_n + 2i \sum_{j=1}^{n-1} s_j z_j
        \bigg),
\]
for every $z \in D_n$. Its expression as a complex vector field is
\[
    X_z^\sharp =
    \sum_{j=1}^{n-1} s_j \Big(\frac{\partial}{\partial z_j}
        + \frac{\partial}{\partial \overline{z}_j} \Big) +
        \bigg(
            s_n + 2i \sum_{j=1}^{n-1} s_j z_j
        \bigg) \frac{\partial}{\partial z_n} +
        \bigg(
            s_n - 2i \sum_{j=1}^{n-1} s_j \overline{z}_j
        \bigg) \frac{\partial}{\partial \overline{z}_n},
\]
for all $z \in D_n$.

\begin{proposition}\label{prop:momentmap_N(n)}
    The function given by
    \begin{align*}
        \mu : D_n &\longrightarrow \R^n \\
        \mu(z) &= -\frac{1}{2(\im(z_n) - |z'|^2)}
            (-4\im(z'), 1)
    \end{align*}
    is a moment map for the action of $N(n)$ on $D_n$.
\end{proposition}
\begin{proof}
    First we observe that the function $\mu$ is easily seen to be $\R^n$-invariant.

    Let $X = (s_1, \dots, s_n) \in \R^n$ be given. Then, the corresponding function $\mu_X : D_n \rightarrow \R$ has the expression
    \begin{align*}
        \mu_X(z) &= \left<\mu(z), X \right> \\
         &= \frac{1}{2(\im(z_n) - |z'|^2)}\bigg(
         \sum_{j=1}^{n-1} 4s_j \im(z_j) - s_n
         \bigg)
    \end{align*}

    As before, we consider the canonical basis $X_1, \dots, X_n$ and prove that $X_j^\sharp = X_{\mu_j}$, for all $j=1, \dots, n$, where $\mu_j = \mu_{X_j}$. In this case we have the functions
    \[
        \mu_j(z) =
        \begin{cases}
            \displaystyle
            \frac{2\im(z_j)}{\im(z_n) - |z'|^2}, & \mbox{if }
                        j=1, \dots, n-1 \\
            \displaystyle
            -\frac{1}{2(\im(z_n) - |z'|^2)}, & \mbox{if } j=n,
        \end{cases}
    \]
    and the vector fields
    \[
        X_j^\sharp(z) =
        \begin{cases}
            \displaystyle
            \frac{\partial}{\partial z_j}
                + \frac{\partial}{\partial \overline{z}_j} +
                2i\bigg(
                    z_j\frac{\partial}{\partial z_n}
                    - \overline{z}_j \frac{\partial}{\partial \overline{z}_n}
                \bigg),
                & \mbox{if } j=1,\dots,n-1 \\
            \displaystyle
            \frac{\partial}{\partial z_n} +
                \frac{\partial}{\partial \overline{z}_n}, & \mbox{if } j=n.
        \end{cases}
    \]
    In particular, the identity
    \[
        X_n^\sharp = X_{\mu_n}
    \]
    follows from the proof of Proposition~\ref{prop:momentmap_P(n)}. Hence, it remains to consider the cases for $j = 1, \dots, n-1$. For these we note that
    \begin{align*}
        \frac{\partial \mu_j}{\partial z_k}(z) &=
        \begin{cases}
            \displaystyle
            \frac{\delta_{jk}}{i(\im(z_n) - |z'|^2)}
            + \frac{2\im(z_j) \overline{z}_k}{(\im(z_n) - |z'|^2)^2},
                & \mbox{if } k=1, \dots, n-1 \\
            \displaystyle
            -\frac{\im(z_j)}{i(\im(z_n) - |z'|^2)^2}, & \mbox{if } j=n,
        \end{cases} \\
        \frac{\partial \mu_j}{\partial \overline{z}_k}(z) &=
        \begin{cases}
            \displaystyle
            -\frac{\delta_{jk}}{i(\im(z_n) - |z'|^2)}
                +\frac{2\im(z_j) z_k}{(\im(z_n) - |z'|^2)^2},
                & \mbox{if }  k=1, \dots, n-1 \\
            \displaystyle
            \frac{\im(z_j)}{i(\im(z_n) - |z'|^2)}, & \mbox{if } j=n.
        \end{cases}
    \end{align*}
    As before, a straightforward replacement in the formula from Lemma~\ref{lem:Hamiltonian_Dn} completes the proof.
\end{proof}

\subsection{Moment map for $N(n,k)$}
Let $X = (s_1, \dots, s_n) \in \R^n$. For the group corresponding to this action we have the exponential
\[
    \exp(sX) = (e^{iss_1}, \dots, e^{iss_k}, ss_{k+1}, \dots, ss_n)
\]
for every $s \in \R$. Its action on any $z \in D_n$ is given by
\begin{multline*}
    \exp(sX)\cdot z =
    \bigg(
        e^{iss_1}z_1, \dots, e^{iss_k}z_k, z_{k+1}+ss_{k+1}, \dots, z_{n-1}+ss_{n-1}, \\
        z_n + ss_n + 2i \sum_{j=k+1}^{n-1} ss_jz_j
        +i \sum_{j=k+1}^{n-1} s^2 s_j^2
    \bigg).
\end{multline*}
Hence, as a function $D_n \rightarrow \C^n$ we have
\[
    X_z^\sharp =
    \bigg(
        is_1z_1, \dots, is_kz_k, s_{k+1}, \dots, s_{n-1},
        s_n + 2i \sum_{j=k+1}^{n-1} s_jz_j
    \bigg),
\]
and the expression of this as a complex vector field is the following
\begin{multline*}
    X_z^\sharp =
        i\sum_{j=1}^k \Big(
            s_j z_j \frac{\partial}{\partial z_j}
            - s_j \overline{z}_j \frac{\partial}{\partial \overline{z}_j}
        \Big) +
        \sum_{j=k+1}^{n-1} s_j \Big(
            \frac{\partial}{\partial z_j}
            + \frac{\partial}{\partial \overline{z}_j}
        \Big) \\
        + \bigg(
            s_n + 2i \sum_{j=k+1}^{n-1} s_j z_j
        \bigg) \frac{\partial}{\partial z_n}
        + \bigg(
            s_n - 2i \sum_{j=k+1}^{n-1} s_j \overline{z}_j
        \bigg) \frac{\partial}{\partial \overline{z}_n}
\end{multline*}
for every $z \in D_n$.

\begin{proposition}\label{prop:momentmap_N(n,k)}
    The function given by
    \begin{align*}
        \mu : D_n &\rightarrow \R^n \\
        \mu(z) &= -\frac{1}{2(\im(z_n) - |z'|^2)}
            (2|z_1|^2, \dots, 2|z_k|^2, -4\im(z_{(2)}), 1)
    \end{align*}
    is a moment map for the action $N(n,k)$ on $D_n$.
\end{proposition}
\begin{proof}
    The function $\mu$ is clearly $\T^k\times\R^{n-k}$-invariant.

    For $X = (s_1, \dots, s_n) \in \R^n$, the corresponding function $\mu_X : D_n \rightarrow \R^n$ is given by
    \begin{align*}
        \mu_X(z) &= \left<\mu(z), X\right> \\
            &= -\frac{1}{2(\im(z_n) - |z'|^2)}
                \bigg(
                    2\sum_{j=1}^k s_j |z_j|^2
                    -4\sum_{j=k+1}^{n-1} s_j \im(z_j)
                    +s_n
                \bigg).
    \end{align*}
    With respect to the canonical basis $X_1, \dots, X_n$, we have the following function $\mu_j = \mu_{X_j}$
    \[
        \mu_j(z) =
        \begin{cases}
            \displaystyle
            -\frac{|z_j|^2}{\im(z_n) - |z'|^2}, & \mbox{if } j=1,\dots,k \\
            \displaystyle
            \frac{2\im(z_j)}{\im(z_n) - |z'|^2}, & \mbox{if } j=k+1,\dots,n-1 \\
            \displaystyle
            -\frac{1}{2(\im(z_n) - |z'|^2)}, & \mbox{if } j=n,
        \end{cases}
    \]
    with corresponding vector fields
    \[
        X_j^\sharp(z) =
        \begin{cases}
            \displaystyle
            iz_j\frac{\partial}{\partial z_j}
            - i\overline{z}_j \frac{\partial}{\partial \overline{z}_j}, & \mbox{if } j=1,\dots,k \\
            \displaystyle
            \frac{\partial}{\partial z_j}
            + \frac{\partial}{\partial \overline{z}_j}
            +2i\bigg(
                z_j\frac{\partial}{\partial z_n}
                -\overline{z}_j\frac{\partial}{\partial \overline{z}_n}
            \bigg), & \mbox{if } j=k+1,\dots,n-1 \\
            \displaystyle
            \frac{\partial}{\partial z_n}
            + \frac{\partial}{\partial \overline{z}_n}, & \mbox{if } j=n.
        \end{cases}
    \]
    As before, it is enough to show that
    \[
        X_j^\sharp = X_{\mu_j}
    \]
    for all $j=1, \dots, n$. This computation is contained in the proof of the previous cases.
\end{proof}

\section{Commuting Toeplitz operators and moment maps. The MASG case.}
\label{sec:commToep_momentmaps_MASG}
It was proved in \cite{QVUnitBall1} and \cite{QVUnitBall2} that every MASG of the group of automorphisms of either $\B^n$ or $D_n$ yields a commutative $C^*$-algebra generated by Toeplitz operators obtained by considering the symbols invariant under the given MASG. We now state this more precisely.

As before, $D$ will denote either of the domains $\B^n$ or $D_n$ with corresponding weighted Bergman spaces $\cA^2_\lambda(D)$, where $\lambda > -1$. For a given MASG $H$ of the group of biholomorphisms of $D$ we will denote by $L^\infty(D)^H$ the space of $H$-invariant essentially bounded symbols on $D$. The next result can be found in \cite{QVUnitBall1}.

\begin{theorem}\label{thm:Toeplitz_H-invariant-QV}
    Let $D$ denote either of the domains $\B^n$ or $D_n$ and let $H$ be a MASG of the group of biholomorphisms of $D$. Then, the $C^*$-algebra $\cT^{(\lambda)}(L^\infty(D)^H)$ generated by Toeplitz operators acting on $\cA^2_\lambda(D)$ with symbols in $L^\infty(D)^H$ is commutative for every $\lambda > -1$.
\end{theorem}

Furthermore, it is possible to write down a unitary transformation that simultaneously diagonalizes all the Toeplitz operators $T^{(\lambda)}_a$, for $a \in L^\infty(D)^H$. Such a unitary map depends only on $H$ and $\lambda$. We refer to \cite{QVUnitBall1} for the specific formulas.

We will reformulate Theorem~\ref{thm:Toeplitz_H-invariant-QV} in terms of moment maps. The key result to achieve this is the following. It establishes that, for a MASG $H$, a symbol is $H$-invariant if and only if it is a function of the moment map of the $H$-action.

\begin{proposition}\label{prop:H-invariance_momentmaps}
    Let $D$ be either of the domains $\B^n$ or $D_n$, and let $H$ be a MASG of the group of biholomorphisms of $D$. If $\mu : D \rightarrow \R^n$ denotes the moment map of the $H$-action, then for every function $a : D \rightarrow \C$ the following conditions are equivalent.
    \begin{enumerate}
        \item The function $a$ is $H$-invariant.
        \item There exists a function $f : \mu(D) \rightarrow \C$ such that the diagram
            \[
                \xymatrix{
                    D \ar[r]^a \ar[d]^\mu & \C \\
                    \mu(D) \ar[ur]^f
                }
            \]
            commutes.
    \end{enumerate}
\end{proposition}
\begin{proof}
    It is enough to prove that the $H$-orbits in $D$ are precisely the fibers of $\mu : D \rightarrow \mu(D)$. Since the map $\mu$ is $H$-invariant, we already know that every fiber of $\mu$ is a union of $H$-orbits. Hence, we only need to show that for $z, w \in D$ such that $\mu(z) = \mu(w)$ there exists $h \in H$ such that $z = h\cdot w$. We will prove this for all the MASG described in Section~\ref{sec:MASG_momentmaps}.

    For the quasi-elliptic case we have $D = \B^n$ and $H = \T^n$. Using the expression for the moment map from Proposition~\ref{prop:momentmap_E(n)}, it is easy to see that the condition $\mu(z) = \mu(w)$ implies $|z_j| = |w_j|$ for all $j = 1, \dots, n$, and so the claim is clear in this case.

    For the rest of the cases, we have $D = D_n$.

    In the quasi-parabolic case, $H = \T^n \times \R$, and by Proposition~\ref{prop:momentmap_P(n)} the condition $\mu(z) = \mu(w)$ is equivalent to
    \[
        \im(z_n) = \im(w_n), \quad |z_j| = |w_j|,
    \]
    for all $j=1, \dots, n-1$. Hence, it is also clear in this case that there exists $(t',h) \in \T^{n-1}\times\R$ such that $z = (t',h)\cdot w$.

    Let us now consider the quasi-hyperbolic case for which $H = \T^{n-1}\times\R_+$. The condition $\mu(z) = \mu(w)$ and Proposition~\ref{prop:momentmap_H(n)} implies in this case that
    \[
        \frac{\re(z_n)}{\im(z_n)-|z'|^2} = \frac{\re(w_n)}{\im(w_n)-|w'|^2},
        \quad
        \frac{|z_j|^2}{\im(z_n)-|z'|^2} = \frac{|w_j|^2}{\im(w_n)-|w'|^2}
    \]
    for all $j = 1, \dots, n-1$. Let us choose the positive real number given by
    \[
        r = \frac{\im(z_n)-|z'|^2}{\im(w_n)-|w'|^2}.
    \]
    Then, it is straightforward to compute from the previous identities that
    \[
        |r^\frac{1}{2} w_j| = |z_j|
    \]
    for all $j = 1, \dots, n-1$. Hence, there exists $t' \in \T^{n-1}$ such that
    \[
        (t',r)\cdot w = (z', rw_n).
    \]
    Note that from the previous identities we have $\re(r w_n) = r \re(w_n) = \re(z_n)$. But we also have
    \begin{align*}
        \im(rw_n) - |z'|^2 &= \im(rw_n) - |r^\frac{1}{2} w'|^2 \\
            &= r(\im(w_n) - |w'|^2) \\
            &= \im(z_n) - |z'|^2.
    \end{align*}
    This implies $rw_n = z_n$ and so we have $z = (t',r)\cdot w$.

    For the nilpotent case we have $H = \R^n$ and the condition $\mu(z) = \mu(w)$ implies, by Proposition~\ref{prop:momentmap_N(n)}, that we have
    \[
        \im(z') = \im(w'), \quad
            \im(z_n) - |z'|^2 = \im(w_n) - |w'|^2.
    \]
    Hence, there exists $b \in \R^{n-1}$ such that
    \begin{align*}
        (b,0)\cdot w &= (z', w_n + 2i w'\cdot b + i|b|^2) \\
            &= (z', (\re(w_n) - 2\im(w')\cdot b)
                + i(\im(w_n) + 2\re(w')\cdot b + |b|^2)),
    \end{align*}
    where the last expression collects the real and imaginary parts of the $n$-th component. Hence, if we take $h = \re(z_n) - \re(w_n) + 2\im(w')\cdot b$, then we have
    \[
        (b,h)\cdot w = (z', \re(z_n) + i(\im(w_n) + 2\re(w')\cdot b + |b|^2)).
    \]

    On the other hand, we note that the function $w \mapsto \im(w_n) - |w'|^2$ is $\R^n$-invariant for this nilpotent action and so its value at $w$ and $(b,h)\cdot w$ is the same. Hence, we have from the previous identities
    \begin{align*}
        \im(z_n) - |z'|^2 &= \im(w_n) - |w'|^2 \\
            &= \im(w_n) + 2\re(w')\cdot b + |b|^2 - |z'|^2.
    \end{align*}
    We conclude that
    \[
        \im(z_n) = \im(w_n) + 2\re(w')\cdot b + |b|^2.
    \]
    This implies that $(b,h)\cdot w = z$.

    For the quasi-nilpotent case $H = \T^k\times\R^{n-k}$ and by Proposition~\ref{prop:momentmap_N(n,k)} the condition $\mu(z) = \mu(w)$ implies that
    \[
        \im(z_{(2)}) = \im(w_{(2)}), \quad
        \im(z_n) - |z'|^2 = \im(w_n) - |w'|^2, \quad
        |z_j| = |w_j|,
    \]
    for all $j=1,\dots,k$. This is clearly equivalent to
    \[
        \im(z_{(2)}) = \im(w_{(2)}), \quad
        \im(z_n) - |z_{(2)}|^2 = \im(w_n) - |w_{(2)}|^2, \quad
        |z_j| = |w_j|,
    \]
    for all $j=1,\dots,k$, and so the result follows from the previous arguments.
\end{proof}

Extending the construction from Proposition~\ref{prop:H-invariance_momentmaps} we consider any connected Abelian subgroup $H$, either maximal or not, of the group of biholomorphisms of $D$. As before, $D$ is either $\B^n$ or $D_n$. Suppose that $\mu : D \rightarrow \fh^*$ is the moment map of the $H$-action on $D$. Then, we are interested in functions of the form $f \circ \mu : D \rightarrow \C$. Note that it is easy to see that $f$ is measurable if and only if $f\circ\mu$ is measurable.

\begin{definition}\label{def:momentmapfunction}
    With the notation from the previous paragraph, we say that a function of the form $f \circ \mu : D \rightarrow \C$ is a moment map function for $H$ or a $\mu$-function on $D$. We denote the set of measurable $\mu$-functions by $L^\infty(D)^\mu$.
\end{definition}

In particular, Proposition~\ref{prop:H-invariance_momentmaps} shows that if $H$ is a MASG, then
\[
    L^\infty(D)^H = L^\infty(D)^\mu.
\]
However, this is not the case for a general Abelian subgroup.

From the previous discussion, we obtain the following alternative formulation of Theorem~\ref{thm:Toeplitz_H-invariant-QV}.

\begin{theorem}\label{thm:Toeplitz_mu-functionsMASG}
    Let $D$ denote either of the domains $\B^n$ or $D_n$ and let $H$ be a MASG of the group of biholomorphisms of $D$. Suppose that $\mu : D \rightarrow \R^n$ denotes the moment map for the $H$-action. Then, the $C^*$-algebra $\cT^{(\lambda)}(L^\infty(D)^\mu)$ generated by Toeplitz operators acting on $\cA^2_\lambda(D)$ whose symbols are $\mu$-functions (i.e.~that belong to $L^\infty(D)^\mu$) is commutative for every $\lambda > -1$.
\end{theorem}

\section{Commuting Toeplitz operators and moment maps. The Abelian subgroup case.}
\label{sec:commToep_momentmaps_Abelian}
As in the previous sections, $D$ will denote either $\B^n$ or $D_n$. From now on $H$ will denote any Abelian subgroup of the group of biholomorphisms of $D$. For simplicity, we will assume that $H$ is connected. There is no real loss of generality since we will be dealing with moment maps where the Lie algebra is the main ingredient. Since $H$ is connected and Abelian, it is a subgroup of a MASG of biholomorphisms of $D$. Hence, after applying a linear fractional transformation of $D$, we can and will assume that $H$ is a subgroup of one of the MASG listed in Section~\ref{sec:MASG_momentmaps}.

Our first observation in this section is the fact that the moment map of the $H$-action can be obtained from the moment map of any MASG containing $H$. To prove this we introduce some notation.

For $H$ a connected Abelian subgroup, let $G$ be a MASG containing $H$. We will assume that $G$ is one of the subgroups listed in Section~\ref{sec:MASG_momentmaps}. Hence, there is a natural inclusion of Lie algebras $\fh \hookrightarrow \fg = \R^n$ that we will denote by $\iota$. This induces a transpose map $\iota^* : \fg^* = \R^n \rightarrow \fh^*$ which is surjective. Note that we have used the fact that $\fg = \R^n$, as described in Section~\ref{sec:MASG_momentmaps}, as well as the identification $\fg^* = \R^n$ given by the canonical inner product on $\R^n$. Furthermore, we can also identify $\fh = \fh^*$ using the inherited inner product, and after doing so the map $\iota^*$ is identified with the orthogonal projection $\R^n \rightarrow \fh$. We will use these identifications in the rest of this work.

\begin{proposition}\label{prop:momentmap_AbelianfromMASG}
    Let $G$ be a MASG of biholomorphisms of $D$ from the list in Section~\ref{sec:MASG_momentmaps} and let $H$ be a connected Abelian subgroup of $G$. Then, the $H$-action on $D$ has a moment map given by
    \begin{align*}
        \mu^H : D &\rightarrow \fh \\
        \mu^H &= \iota^* \circ \mu^G,
    \end{align*}
    where $\mu^G : D \rightarrow \R^n$ is the moment map for the $G$-action on $D$ and $\iota^* : \R^n \rightarrow \fh$ is the orthogonal projection.
\end{proposition}
\begin{proof}
    By definition, for every $X \in \R^n$, the function $\mu^G_X : D \rightarrow \R$ given by
    \[
        \mu^G_X(z) = \langle \mu^G(z), X \rangle,
    \]
    has Hamiltonian vector field $X^\sharp$. Note that, for the identifications described before, the pairing $\langle\cdot,\cdot\rangle$ is the canonical inner product on $\R^n$. In particular, this holds for every $X \in \fh$. In this case, we have
    \[
        \langle \iota^* \circ \mu^G(z) , X\rangle
        = \langle \mu^G(z) , X\rangle
    \]
    because $\iota^*$ is the orthogonal projection $\R^n \rightarrow \fh$ and $X \in \fh$. And so $\iota^*\circ\mu^G$ yields Hamiltonian vector fields corresponding to $X \in \fh$.

    Finally, the $G$-invariance of $\mu^G$ trivially implies the $H$-invariance of $\iota^*\circ\mu^G$. Hence, we conclude that $\mu^H = \iota^*\circ\mu^G$ is a moment map for the $H$-action.
\end{proof}

\begin{remark}\label{rmk:Hmomentfunctions_Gmomentfunctions}
    With the notation of Proposition~\ref{prop:momentmap_AbelianfromMASG}, we observe that for any function $f : \mu^H(D) \rightarrow \C$ we have
    \[
        f\circ \mu^H = f\circ\iota^*\circ\mu^G.
    \]
    Hence, any moment map function for the group $H$ is a moment map function for the group $G$. In other words, we have the inclusion of symbol sets
    \[
        L^\infty(D)^{\mu^H} \subset L^\infty(D)^{\mu^G} = L^\infty(D)^G,
    \]
    where the last identity follows from Proposition~\ref{prop:H-invariance_momentmaps} since $G$ is a MASG. We can add the observation
    \[
        L^\infty(D)^G \subset L^\infty(D)^H
    \]
    that follows trivially from $H \subset G$.
\end{remark}

The following result shows how to obtain a commutative $C^*$-algebra from any connected Abelian subgroup of biholomorphisms.

\begin{theorem}\label{thm:Toeplitz_mu-functions}
    Let $D$ be either of the domains $\B^n$ or $D_n$ and let $H$ be a connected Abelian subgroup of the group of biholomorphisms of $D$ with moment map $\mu$. Then, for every $\lambda > -1$, the $C^*$-algebra $\cT^{(\lambda)}(L^\infty(D)^\mu)$ generated by Toeplitz operators whose symbols are essentially bounded $\mu$-functions is commutative. Furthermore, if $G$ is a MASG containing $H$, then we have the inclusion
    \[
        \cT^{(\lambda)}(L^\infty(D)^\mu)
            \subset \cT^{(\lambda)}(L^\infty(D)^G)
    \]
    for every $\lambda > -1$, where $\cT^{(\lambda)}(L^\infty(D)^G)$ is the $C^*$-algebra generated by Toeplitz operators with $G$-invariant symbols.
\end{theorem}
\begin{proof}
    By Proposition~\ref{prop:momentmap_AbelianfromMASG}, we can assume that the moment map of the $H$-action is given by $\mu = \iota^*\circ\mu^G$. Hence, Remark~\ref{rmk:Hmomentfunctions_Gmomentfunctions} implies the inclusions
    \[
        L^\infty(D)^\mu \subset L^\infty(D)^{\mu^G} = L^\infty(D)^G
    \]
    which imply that
    \[
        \cT^{(\lambda)}(L^\infty(D)^\mu)
            \subset \cT^{(\lambda)}(L^\infty(D)^G)
    \]
    for every $\lambda > -1$. Then, the commutativity of $\cT^{(\lambda)}(L^\infty(D)^G)$ (see \cite{QVUnitBall1}) implies the commutativity of $\cT^{(\lambda)}(L^\infty(D)^\mu)$.
\end{proof}

Next, we observe the following behavior of the commutative $C^*$-algebras obtained from Theorem~\ref{thm:Toeplitz_mu-functions}.

\begin{theorem}\label{thm:Toeplitz-indexedbyH}
    Let $D$ denote either of the domains $\B^n$ or $D_n$. Then, for every $\lambda > -1$ the assignment
    \[
        H \mapsto \cT^{(\lambda)}(L^\infty(D)^{\mu^H})
    \]
    is an inclusion preserving map from the set of connected Abelian subgroups of the group of biholomorphisms of $D$ into the set of commutative $C^*$-algebras of operators acting on $\cA^2_\lambda(D)$.
\end{theorem}
\begin{proof}
    Let $H_1$ and $H_2$ be connected Abelian subgroups of the group of biholomorphisms of $D$, and assume that $H_1 \subset H_2$. Then, we know that the moment maps of their actions are given by
    \[
        \mu^{H_1} = \iota_1^* \circ \mu^G, \quad
        \mu^{H_2} = \iota_2^* \circ \mu^G,
    \]
    where $G$ is a MASG containing $H_2$ and $\iota_j^* : \R^n \rightarrow \fh_j$ are the orthogonal projections, for $j =1, 2$. In particular, if we denote by $\pi : \fh_2 \rightarrow \fh_1$ the orthogonal projection, then we have
    \[
        \mu^{H_1} = \pi \circ \mu^{H_2}.
    \]
    This implies that every moment map function for $H_1$ is a moment map function for $H_2$. In other words, we have
    \[
        L^\infty(D)^{\mu^{H_1}} \subset L^\infty(D)^{\mu^{H_2}}.
    \]
    This and Theorem~\ref{thm:Toeplitz_mu-functions} prove the result.
\end{proof}

\begin{remark}
    By the previous result and for $\lambda > -1$ fixed, the family of commutative $C^*$-algebras
    \[
        \Big\{ \cT^{(\lambda)}(L^\infty(D)^{\mu^H}) \mid
                H \text{ Abelian subgroup}  \Big\}
    \]
    where $H$ runs through all the connected Abelian subgroups acting holomorphically on $D$, is naturally ordered by such subgroups $H$. If $G$ is a MASG, then its corresponding $C^*$-algebra contains that of any subgroup $H$ of $G$. On the other extreme, for the trivial identity group, the moment map is the zero map and so the corresponding moment map functions are the constant ones. Hence, the $C^*$-algebra associated to the trivial identity group consists of the multiples of the identity operator acting on $\cA^2_\lambda(D)$. It is interesting to compare this situation with the one corresponding to the $H$-invariant symbols. In fact, an inclusion $H_1 \subset H_2$ of connected Abelian subgroups acting holomorphically on $D$ implies a reversed inclusion
    \[
        L^\infty(D)^{H_2} \subset L^\infty(D)^{H_1},
    \]
    which yields the reversed inclusion
    \[
        \cT^{(\lambda)}(L^\infty(D)^{H_2}) \subset
            \cT^{(\lambda)}(L^\infty(D)^{H_1}).
    \]
    The effect of this is that even if $H_2$ yields a commutative $C^*$-algebra when we use $H_2$-invariant symbols, the $C^*$-algebra obtained from a subgroup $H_1$, using $H_1$-invariant symbols, is not necessarily commutative. We refer to Example~6.5 from \cite{DOQJFA} for the description of a case where this occurs.
\end{remark}

We now proceed to describe the spaces of symbols obtained as moment map functions. As noted above, we can assume that $H$ is a connected subgroup of one of the MASG described in Section~\ref{sec:MASG_momentmaps}. So we start by describing the connected subgroups of the MASG and their moment maps.

We recall that for a given connected Lie group $G$, there is a one-to-one correspondence between the connected subgroups of $G$ and the Lie subalgebras $\fh$ of the Lie algebra $\fg$ of $G$. In our situation, $G$ is a MASG of the group of biholomorphisms of $D$, and so $\fg = \R^n$ as described before. Because $\fg = \R^n$ is Abelian, every subspace is a Lie subalgebra and the subgroup corresponding to a subspace is precisely the image of the subspace under the exponential map. This yields the next result.

\begin{proposition}\label{prop:Abeliansubgroups}
    Let $G$ be one of the MASGs described in Section~\ref{sec:MASG_momentmaps}. If $\fh \subset \R^n$ is any subspace, then the subset of $G$ given by
    \[
        H = \exp(\fh)
    \]
    is a connected Abelian subgroup. Furthermore, every connected Abelian subgroup of $G$ is obtained in this way.
\end{proposition}

\begin{remark}\label{rmk:subgroups_and_orthogonalsets}
    For every MASG $G$ from Section~\ref{sec:MASG_momentmaps},  Proposition~\ref{prop:Abeliansubgroups} can be rephrased saying that the assignment
    \[
        \fh \mapsto \exp(\fh)
    \]
    is a bijection from the collection of subspaces of $\R^n$ onto the set of connected Abelian subgroups of $G$. For our purposes, a better way to parameterize such subgroups is by using coordinates. Hence, we will rather make use of the assignment
    \[
        \beta \mapsto \exp(\R\langle\beta\rangle)
    \]
    which is a surjection from the collection of linearly independent subsets of $\R^n$ onto the set of connected Abelian subgroups of $G$. Here and from now on, $\R\langle\beta\rangle$ denotes the subspace spanned by $\beta$. If the linearly independent set $\beta$ is orthogonal, then the orthogonal projection $\iota^* : \R^n \rightarrow \fh = \R\langle\beta\rangle$ is given by
    \[
        \iota^*(s) = \sum_{j=1}^m
            \frac{\langle s, v_j \rangle}{\langle v_j,v_j \rangle} v_j
    \]
    where $\beta = \{ v_1, \dots, v_m\}$.
\end{remark}

The previous discussion allows us to write down the following general formula for the moment maps of the connected Abelian subgroups under discussion. The proof is easily obtained from Proposition~\ref{prop:momentmap_AbelianfromMASG} and Remark~\ref{rmk:subgroups_and_orthogonalsets}.

\begin{proposition}\label{prop:momentmap_AbelianfromMASG-basis}
    Let $G$ be a MASG from the list in Section~\ref{sec:MASG_momentmaps} and $H$ a connected Abelian subgroup of $G$. Let $\beta = \{ v_1, \dots, v_m\}$ be an orthogonal basis of the Lie algebra $\fh$ of $H$. Then, the moment map for the $H$-action on $D$ (either $\B^n$ or $D_n$) is the map $\mu^H : D \rightarrow \fh$ given by
    \[
        \mu^H(z) =
            \sum_{j=1}^m \frac{\langle \mu^G(z), v_j\rangle}{\langle v_j,v_j \rangle} v_j
    \]
    for every $z \in D$.
\end{proposition}

We now obtain the following description of the moment map functions for connected Abelian subgroups $H$ as above.

\begin{corollary}\label{cor:momentmapfunctions_betaG}
    With the notation of Proposition~\ref{prop:momentmap_AbelianfromMASG-basis}, let us choose $\beta = \{ v_1, \dots, v_m\}$ any basis for $\fh$. Then, the space of essentially bounded moment map functions for $H$ is given by
    \[
        L^\infty(D)^{\mu^H} =
        \{ f(a_1, \dots, a_m) \mid f \in L^\infty((a_1, \dots, a_m)(\B^n))\},
    \]
    where the functions $a_j : \B^n \rightarrow \R$ are defined by
    \[
        a_j(z) = \langle \mu^G(z), v_j\rangle
    \]
    for all $z \in D$ and for every $j = 1, \dots, m$.
\end{corollary}
\begin{proof}
    Let $\beta' = \{v'_1, \dots, v'_m\}$ be an orthogonal basis of $\fh$. Then, by Definition~\ref{def:momentmapfunction} and Proposition~\ref{prop:momentmap_AbelianfromMASG} the moment map functions for $H$ are precisely functions of the form
    \[
        a(z) = f(\langle\mu^G(z),v'_1\rangle, \dots,
        \langle\mu^G(z),v'_m\rangle)
    \]
    for all $z \in D$ and some function $f$. Since there is an invertible matrix $(c_{jk})_{j,k=1}^m$ such that
    \[
        \langle\mu^G(z),v_j\rangle = \sum_{j,k=1}^m c_{jk}
                \langle\mu^G(z),v'_k\rangle
    \]
    the result is now clear.
\end{proof}

\begin{remark}\label{rmk:momentmapfunctions_beta_matrices}
    Still with the notation of Proposition~\ref{prop:momentmap_AbelianfromMASG-basis} let us fix a MASG $G$. For a basis $\beta = \{v_1, \dots, v_m\}$, not necessarily orthogonal, of the Lie algebra $\fh$ of a connected subgroup $H$ of $G$, let us consider the matrix $A(\beta)$ whose rows are the elements of $\beta$, in other words, we take
    \[
        A(\beta) =
        \begin{pmatrix}
          v_1 \\
          \vdots \\
          v_m
        \end{pmatrix}.
    \]
    Then, for the functions $a_1, \dots, a_m$ defined in Corollary~\ref{cor:momentmapfunctions_betaG} we have
    \[
        \begin{pmatrix}
          a_1(z) \\
          \vdots \\
          a_m(z)
        \end{pmatrix} = A(\beta) \mu^G(z)^\top,
    \]
    for all $z \in D$. It follows that the essentially bounded moment map functions for $H$ are functions $a : D \rightarrow \C$ of the form
    \[
        a(z) = f(A(\beta)\mu^G(z)^\top)
    \]
    for all $z \in D$, for a suitable function $f$. Note that in this expression we are writing the arguments of $f$ as a column instead of a row. There should not be any confusion with this slight abuse of notation, which we will use in the rest of this work. On the other hand, if $A$ is a real $m\times n$ matrix with rank $m$ the functions $a : D \rightarrow \C$ of the form
    \[
        a(z) = f(A\mu^G(z)^\top)
    \]
    for all $z \in D$ and a suitable $f$, are moment map functions. The connected Abelian subgroup $H$ for which this holds is the connected subgroup of $G$ whose Lie algebra $\fh$ is the row space of the matrix $A$. Hence, there is a correspondence between the different types of moment map functions obtained from Corollary~\ref{cor:momentmapfunctions_betaG} and rank $m$ real matrices of size $m\times n$ ($1 \leq m \leq n$). This correspondence is not one-to-one since two such matrices may have the same row space. With respect to this correspondence, we observe that for $H = G$, a MASG, we can choose $\beta$ to be any orthogonal basis of $\R^n$ to obtain the moment map $\mu^G$. Recall that in this case, by Proposition~\ref{prop:H-invariance_momentmaps} it follows that the moment map functions for $G$ are precisely the $G$-invariant functions.
\end{remark}

The next step is to apply the expressions of the moment maps of the MASG computed in Section~\ref{sec:MASG_momentmaps} to obtain the explicit formulas for the moment map functions for connected Abelian subgroups. This will be done in the next subsections for each MASG. This is used to show that the symbols introduced in this work include many of those already known in the literature, but they also include new examples not considered previously. We also obtain corresponding commutativity results for $C^*$-algebras generated by Toeplitz operators with our special symbols.

\subsection{$\beta$-quasi-elliptic symbols}
In this subsection we consider the case $G = \T^n$ for the quasi-elliptic action $E(n)$ described in Section~\ref{sec:MASG_momentmaps}. The next result provides the explicit expression of the moment maps for connected subgroups of $\T^n$. It is a consequence of Propositions~\ref{prop:momentmap_E(n)} and \ref{prop:momentmap_AbelianfromMASG-basis} and Corollary~\ref{cor:momentmapfunctions_betaG}.

\begin{proposition}\label{prop:momentmap_betaE(n)}
    Let $H$ be a connected Abelian subgroup of $\T^n$ whose Lie algebra $\fh$ has the basis $\beta =  \{v_1, \dots, v_m\}$. Then, the following hold.
    \begin{enumerate}
        \item If $\beta$ is orthogonal, then the moment map for the $H$-action on $\B^n$ is the map $\mu^H : \B^n \rightarrow \fh$ given by
            \[
                \mu^H(z) = -\frac{1}{1-|z|^2}
                    \sum_{j=1}^m
                    \frac{\langle (|z_1|^2, \dots, |z_n|^2), v_j\rangle}{\langle v_j,v_j \rangle} v_j,
            \]
            for every $z \in \B^n$.
        \item If $\beta$ is an arbitrary basis, then the space of essentially bounded moment map functions for $H$ is given by
            \[
                L^\infty(\B^n)^{\mu^H} =
                \{ f(a_1, \dots, a_m) \mid f \in L^\infty((a_1, \dots, a_m)(\B^n))\},
            \]
            where $a_j : \B^n \rightarrow \C$ are defined by
            \[
                a_j(z) = \frac{1}{1-|z|^2}\langle (|z_1|^2, \dots, |z_n|^2), v_j\rangle,
            \]
            for all $z \in \B^n$ and for all $j = 1, \dots, m$.
    \end{enumerate}
\end{proposition}

The moment map functions for the subgroups in the previous result yield the following special symbols.

\begin{definition}\label{def:beta-quasi-elliptic}
    Let $\beta = \{ v_1, \dots, v_m \} \subset \R^n$ be a linearly independent set. An essentially bounded function $a : \B^n \rightarrow \C$ is called a $\beta$-quasi-elliptic symbol if there is a measurable function $f$ such that
    \[
        a = f(a_1, \dots, a_m),
    \]
    where $a_1, \dots, a_m$ are given as in Proposition~\ref{prop:momentmap_betaE(n)}. The space of all essentially bounded $\beta$-quasi-elliptic symbols will be denoted by $L^\infty(\B^n)_{\beta,E(n)}$. In other words, we have
    \[
        L^\infty(\B^n)_{\beta,E(n)} = L^\infty(\B^n)^{\mu^H},
    \]
    where $H = \exp(\R\langle\beta\rangle)$.
\end{definition}

As noted in Remark~\ref{rmk:momentmapfunctions_beta_matrices}, for $H = \T^n$ and for any basis $\beta$ of $\R^n$, the $\beta$-quasi-elliptic symbols are precisely the quasi-elliptic symbols from \cite{QVUnitBall1}.

\begin{example}[Quasi-radial symbols as $\beta$-quasi-elliptic symbols]
\label{ex:quasi-radial_partitions}
    Let us consider a partition $k = (k_1, \dots, k_m) \in \Z_+^m$ of $n$. In other words, $n = k_1 + \dots + k_m$ and $k_j \geq 1$ for all $j = 1, \dots, m$. For every $z \in \C^n$ we have a corresponding decomposition
    \[
        z = (z_{(1)}, \dots, z_{(m)})
    \]
    where $z_{(j)} \in \C^{k_j}$ is given by
    \[
        z_{(j)} = (z_{k_0 + \dots + k_{j-1}+1}, \dots, z_{k_0 + \dots + k_j}).
    \]
    We will use the convention $k_0 = 0$. Recall from \cite{VasilevskiQuasiRadial2010} that a symbol on $a : \B^n \rightarrow \C$ is called $k$-quasi-radial if it can be written as
    \[
        a(z) = f(|z_{(1)}|, \dots, |z_{(m)}|)
    \]
    for some function $f$. Let us now consider the group
    \[
        T_k = \{ (t_1 1_{k_1}, \dots, t_m 1_{k_m}) \mid t \in \T^m \}
            \simeq \T^m,
    \]
    where we denote $1_{k_j} = (1, \dots, 1) \in \C^{k_j}$ for every $j = 1, \dots, m$. The group $T_k$ is a connected closed subgroup of $\T^n$ whose Lie algebra has the orthogonal basis $\beta(k)$ that consists of the elements given by (for $j = 1, \dots, m$)
    \[
        v_j = (0, 1_{k_j}, 0)
    \]
    where the non-zero entries occur exactly at the indices $k_0 + \dots + k_{j-1}+1, \dots, k_0 + \dots + k_j$. We note that the corresponding matrix $A(\beta)$ obtained from Remark~\ref{rmk:momentmapfunctions_beta_matrices} is given by
    \[
        A(\beta) =
        \begin{pmatrix}
          1_{k_1} & 0 & \cdots & 0 \\ 0 & 1_{k_2} & \cdots & 0 \\
          \vdots & \vdots & \vdots & \vdots \\
          0 & 0 & \cdots & 1_{k_m}
        \end{pmatrix},
    \]
    which is a very simple example of a row-reduced echelon matrix. Note that we have
    \[
        |z_{(j)}|^2 = \langle(|z_1|^2, \dots, |z_n|^2),v_j\rangle,
    \]
    for every $j=1, \dots, m$. On the other hand, it is easy to see that the map
    \[
        (s_1, \dots, s_m) \mapsto \frac{1}{1-(s_1+\dots+s_m)}
            (s_1, \dots, s_m)
    \]
    can be inverted smoothly when $s_1 + \dots + s_m < 1$ and $s_j \geq 0$ for all $j = 1, \dots, m$. We conclude that for a symbol $a : \B^n \rightarrow \C$ the next conditions are equivalent
    \begin{itemize}
        \item $a(z) = f_1(|z_{(1)}|, \dots, |z_{(m)}|)$ for some $f_1$.
        \item $a(z) = f_2\bigg(\frac{|z_{(1)}|^2}{1-|z|^2},
                \dots,
                \frac{|z_{(m)}|^2}{1-|z|^2} \bigg)$ for some $f_2$.
    \end{itemize}
    It follows that, for this particular choice of $\beta(k)$, the $\beta(k)$-quasi-elliptic symbols from Definition~\ref{def:beta-quasi-elliptic} are precisely the $k$-quasi-radial symbols from \cite{VasilevskiQuasiRadial2010}. We also note that for the partition $k=(n)$ we obtain the radial symbols (see \cite{GQVJFA}) and for the partition $k=(1, \dots, 1) \in \Z^n_+$ we obtain the quasi-elliptic or separately radial symbols (see \cite{QVUnitBall1}). In particular, all of these spaces of symbols, $k$-quasi-radial, radial and separately radial, are realized as $\beta$-quasi-elliptic symbols for suitable choices of $\beta$.
\end{example}

By the previous example, the next result generalizes the commutativity of Toeplitz operators with $k$-quasi-radial symbols stated in \cite{VasilevskiQuasiRadial2010}. This result is a consequence of Theorem~\ref{thm:Toeplitz_mu-functions}.

\begin{theorem}\label{thm:commToeplitz-beta-quasi-elliptic}
    Let $\beta = \{ v_1, \dots, v_m\} \subset \R^n$ be a linearly independent set. Then, for every $\lambda > -1$ the $C^*$-algebra $\cT^{(\lambda)}(L^\infty(\B^n)_{\beta,E(n)})$ generated by Toeplitz operators with essentially bounded $\beta$-quasi-elliptic symbols acting on $\cA^2_\lambda(\B^n)$ is commutative.
\end{theorem}

The next result shows that the commutative $C^*$-algebras obtained in Theorem~\ref{thm:commToeplitz-beta-quasi-elliptic} come, in some cases, from spaces of symbols not considered before. The particular case given in the statement can be generalized easily to obtain many other examples using the arguments of its proof.

\begin{proposition}\label{prop:betaquasielliptic-notkquasi-radial}
    If $n \geq 2$ and for $\beta = \{ e_1\}$, the space of $\beta$-quasi-elliptic symbols is different from the space of $k$-quasi-radial symbols defined by any partition $k$ of $n$. In other words, we have
    \[
        L^\infty(\B^n)_{\beta,E(n)} \neq L^\infty(\B^n)_{\beta(k),E(n)}
    \]
    for every partition $k$ of $n$, where $\beta(k)$ is defined in Example~\ref{ex:quasi-radial_partitions}.
\end{proposition}
\begin{proof}
    We first consider the partition $k = (1,n-1)$. Following the notation from Example~\ref{ex:quasi-radial_partitions}, to this partition corresponds the basis $\beta(k)$ that consists of the vectors
    \[
        v_1 = e_1, \quad v_2 = (0,1, \dots, 1).
    \]
    The $\beta$-quasi-elliptic symbols $a : \B^n \rightarrow \C$ are of the form
    \[
        a(z) = f\bigg( \frac{|z_1|^2}{1-|z|^2}\bigg)
    \]
    for some function $f$. Note that we can assume that $f$ is defined in $[0,\infty)$.

    On the other hand, by Example~\ref{ex:quasi-radial_partitions}, the $k$-quasi-radial symbols $a : \B^n \rightarrow \C$ are of the form
    \[
        a(z) = f(|z_1|, |z_{(2)}|)
    \]
    for some function $f$. In particular, for these choices of $\beta$ and $k$ every $\beta$-quasi-elliptic symbol is a $k$-quasi-radial symbol, so we have
    \[
        L^\infty(\B^n)_{\beta,E(n)} \subset L^\infty(\B^n)_{\beta(k),E(n)}.
    \]
    But it turns out that this inclusion is proper. To prove this consider the symbol $a_2 : \B^n \rightarrow \C$ defined by
    \[
        a_2(z) = |z_{(2)}|.
    \]
    Clearly, $a_2$ belongs to $L^\infty(\B^n)_{\beta(k),E(n)}$. If we assume the existence of a function $f : [0,\infty) \rightarrow \C$ such that
    \[
        |z_{(2)}| = a_2(z) = f\bigg( \frac{|z_1|^2}{1-|z|^2}\bigg)
    \]
    for all $z \in \B^n$, then by taking $z$ such that $z_{(2)} = 0$ we conclude
    \[
        0 = f\bigg( \frac{|z_1|^2}{1-|z_1|^2}\bigg)
    \]
    for every $z_1 \in \B^1$. But then we have $f \equiv 0$, which is a contradiction. This implies that
    \[
        a_2 \in L^\infty(\B^n)_{\beta(k),E(n)} \setminus
                L^\infty(\B^n)_{\beta,E(n)},
    \]
    thus proving the proper inclusion
    \[
        L^\infty(\B^n)_{\beta,E(n)} \subsetneq L^\infty(\B^n)_{\beta(k),E(n)}.
    \]

    It is a simple exercise to adapt this argument to show that for any $k \in \Z^m_+$ with $m \geq 2$ we have
    \[
        L^\infty(\B^n)_{\beta,E(n)} \neq L^\infty(\B^n)_{\beta(k),E(n)}.
    \]
    Finally, for the partition $k = (n)$, if the symbol given by $z \mapsto |z|$ belongs to $L^\infty(\B^n)_{\beta,E(n)}$, then it is not difficult to see that the symbol $z \mapsto |z_1|$ also belongs to $L^\infty(\B^n)_{\beta,E(n)}$. But this implies that the symbol $z \mapsto |(0, z_2, \dots, z_n)|$ is an element of $L^\infty(\B^n)_{\beta,E(n)}$ as well, which contradicts the previous arguments. This completes the proof that
    \[
        L^\infty(\B^n)_{\beta,E(n)} \neq L^\infty(\B^n)_{\beta(k),E(n)}
    \]
    for every partition $k$ of $n$.
\end{proof}

\subsection{$\beta$-quasi-parabolic symbols}
\label{subsec:beta-quasi-parabolic}
We now consider the group $G = \T^{n-1}\times \R$ acting by the quasi-parabolic action $P(n)$ on $D_n$ defined in Section~\ref{sec:MASG_momentmaps}. As before, we will consider some linearly independent set $\beta = \{v_1, \dots, v_m\} \subset \R^n$, and following Remark~\ref{rmk:momentmapfunctions_beta_matrices} we have a corresponding matrix $A(\beta)$ whose rows are $v_1, \dots, v_m$. Then, we have the group $H$ whose Lie algebra is $\fh = \R\langle\beta\rangle$. If the last column of $A(\beta)$ is zero, then $H$ is a subgroup of $\T^{n-1} \subset \T^{n-1}\times\R$. In this case, $H$ is in fact a subgroup of a compact MASG and, up to conjugacy of Lie groups, this has been discussed in the description of $\beta$-quasi-elliptic symbols. Hence, we can assume without lost of generality that the last column of $A(\beta)$ is non-zero as long as the case corresponding to $\beta$-quasi-elliptic symbols has already been considered.

The next result exhibits the moment maps for connected subgroups of $\T^{n-1}\times\R$. It is a consequence of Propositions~\ref{prop:momentmap_P(n)} and \ref{prop:momentmap_AbelianfromMASG-basis} and Corollary~\ref{cor:momentmapfunctions_betaG}.

\begin{proposition}\label{prop:momentmap_betaP(n)}
    Let $H$ be a connected Abelian subgroup of $\T^{n-1}\times\R$ whose Lie algebra $\fh$ has the basis $\beta = \{ v_1, \dots, v_m \}$. Then, the following hold.
    \begin{enumerate}
        \item If $\beta$ is orthogonal, then the moment map for the $H$-action on $D_n$ is the map $\mu^H : D_n \rightarrow \fh$ given by
            \[
                \mu(z) = -\frac{1}{2(\im(z_n)-|z'|^2)} \sum_{j=1}^m
                    \frac{\langle(2|z_1|^2, \dots, 2|z_{n-1}|^2, 1),v_j\rangle}{\langle v_j,v_j\rangle} v_j,
            \]
            for every $z \in D_n$.
        \item If $\beta$ is an arbitrary basis, then the space of essentially bounded moment map functions for $H$ is given by
            \[
                L^\infty(D_n)^{\mu^H} = \{ f(a_1, \dots, a_m) \mid
                    f \in L^\infty((a_1, \dots, a_m)(D_n)) \},
            \]
            where $a_j : D_n \rightarrow \R$ are defined by
            \[
                a_j(z) = \frac{1}{2(\im(z_n)-|z'|^2)}
                    \langle(2|z_1|^2, \dots, 2|z_{n-1}|^2, 1),v_j\rangle,
            \]
            for every $z \in D_n$ and for every $j = 1, \dots, m$.
    \end{enumerate}
\end{proposition}

The moment map functions from the previous result yield the following special symbols.

\begin{definition}\label{def:beta-quasi-parabolic}
    Let $\beta = \{ v_1, \dots, v_m \} \subset \R^n$ be a linearly independent set. An essentially bounded function $a : D_n \rightarrow \C$ is called a $\beta$-quasi-parabolic symbol if there is a measurable function $f$ such that
    \[
        a = f(a_1, \dots, a_m),
    \]
    where $a_1, \dots, a_m$ are given as in Proposition~\ref{prop:momentmap_betaP(n)}. The space of all essentially bounded $\beta$-quasi-parabolic symbols will be denoted by $L^\infty(D_n)_{\beta,P(n)}$. In other words, we have
    \[
        L^\infty(D_n)_{\beta,P(n)} = L^\infty(D_n)^{\mu^H},
    \]
    where $H = \exp(\R\langle\beta\rangle)$.
\end{definition}

\begin{example}[Parabolic quasi-radial symbols as $\beta$-quasi-parabolic symbols]
\label{ex:beta-quasi-parabolic_partitions}
    Let us consider a partition $k \in \Z_+^m$ with $m \geq 2$ such that $k_m = 1$, so we can write $k = (k',1)$ where $k' \in \Z_+^{m-1}$ is a partition of $n-1$. Following the notation of Example~\ref{ex:quasi-radial_partitions} we define the linearly independent set $\beta(k) = \{v_1, \dots, v_m\}$ where we choose
    \[
        v_j = (0, 1_{k_j}, 0)
    \]
    for every $j = 1, \dots, m$, where the non-zero entries occur exactly at the indices $k_0 + \dots + k_{j-1}+1, \dots, k_0 + \dots + k_j$. In particular, we have $v_m = e_n$ since $k_m = 1$. In the notation of Remark~\ref{rmk:momentmapfunctions_beta_matrices}, the matrix associated to $\beta(k)$ is given by
    \[
        A(\beta(k)) =
        \begin{pmatrix}
          1_{k_1} & 0 & \cdots & 0 & 0 \\
          0 & 1_{k_2} & \cdots & 0  & 0 \\
          \vdots & \vdots & \vdots & \vdots & \vdots \\
          0 & 0 & \cdots & 1_{k_{m-1}} & 0 \\
          0 & 0 & \cdots & 0 & 1
        \end{pmatrix}.
    \]
    For the basis $\beta(k)$ just defined, the functions $a_j$ from Proposition~\ref{prop:momentmap_betaP(n)} are given by
    \[
        a_j(z) =
        \begin{cases}
            \displaystyle
            \frac{|z_{(k_j)}|^2}{\im(z_n)-|z'|^2},
                    & \mbox{if } j=1, \dots, m-1 \\
            \displaystyle
            \frac{1}{2(\im(z_n)-|z'|^2)}, & \mbox{if } j = m,
        \end{cases}
    \]
    for all $z \in D_n$. Using a simple change of variables, it follows that the $\beta(k)$-quasi-parabolic symbols are exactly the functions of $(|z_{(1)}|, \dots, |z_{(m-1)}|, \im(z_n))$. More precisely, we have
    \[
        a \in L^\infty(D_n)_{\beta(k),P(n)} \iff
        a(z) = f(|z_{(1)}|, \dots, |z_{(m-1)}|, \im(z_n)),
    \]
    for some essentially bounded function $f$. It follows that, for this particular choice of $\beta(k)$, the $\beta(k)$-quasi-parabolic symbols from Definition~\ref{def:beta-quasi-parabolic} are precisely the parabolic $k$-quasi-radial symbols from \cite{VasilevskiParabolicQuasiRadial2010}.
\end{example}

By the previous example the next result generalizes the commutativity of Toeplitz operators with quasi-parabolic symbols stated in \cite{QVUnitBall1} and with parabolic $k$-quasi-radial symbols stated in \cite{VasilevskiParabolicQuasiRadial2010}. This result is a consequence of Theorem~\ref{thm:Toeplitz_mu-functions}.

\begin{theorem}\label{thm:commToeplitz-beta-quasi-parabolic}
    Let $\beta = \{ v_1, \dots, v_m\} \subset \R^n$ be a linearly independent set. Then, for every $\lambda > -1$ the $C^*$-algebra $\cT^{(\lambda)}(L^\infty(D_n)_{\beta,P(n)})$ generated by Toeplitz operators with essentially bounded $\beta$-quasi-parabolic symbols acting on $\cA^2_\lambda(D_n)$ is commutative.
\end{theorem}

As in the $\beta$-quasi-elliptic case, we now prove that for some $\beta$ the $\beta$-quasi-parabolic symbols are not given as a set of parabolic $k$-quasi-radial symbols. Again, the argument shows that many other examples are possible.

\begin{proposition}\label{prop:betaquasiparabolic-notparabolickquasi-radial}
    Assume that $n \geq 3$. For $\beta = \{ e_1, e_n\}$, the space of $\beta$-quasi-parabolic symbols is different from the space of parabolic $k$-quasi-radial symbols defined by any partition $k$ of $n$ of the form $k = (k',1)$. In other words, we have
    \[
        L^\infty(D_n)_{\beta,P(n)} \neq L^\infty(D_n)_{\beta(k),P(n)}
    \]
    for every partition $k = (k',1)$ of $n$ such that $k'$ is a partition of $n-1$, where $\beta(k)$ is defined in Example~\ref{ex:beta-quasi-parabolic_partitions}.
\end{proposition}
\begin{proof}
    Let us first consider the partition $k = (1,n-2,1)$. Following the notation from Example~\ref{ex:beta-quasi-parabolic_partitions}, this partition has a corresponding basis $\beta(k)$ that consists of the vectors
    \[
        v_1 = e_1, \quad v_2 = (0,1, \dots, 1, 0), \quad v_3 = e_n.
    \]
    By Proposition~\ref{prop:momentmap_betaP(n)} we see that $\beta$-quasi-parabolic symbols are functions of the expression
    \[
        \bigg(
            \frac{|z_1|^2}{\im(z_n)-|z'|^2}, \im(z_n)-|z'|^2
        \bigg).
    \]
    Hence, a symbol $a : D_n \rightarrow \C$ is $\beta$-quasi-parabolic if and only if it is of the form
    \[
        a(z) = f(|z_1|, \im(z_n) - |z'|^2)
    \]
    for some function $f$. Note that we can assume that $f$ is defined in $[0,\infty)^2$.

    On the other hand, by Example~\ref{ex:quasi-radial_partitions} the parabolic $k$-quasi-radial symbols $a : D_n \rightarrow \C$ are of the form
    \[
        a(z) = f(|z_1|, |z_{(2)}|, \im(z_n))
    \]
    for some function $f$. In particular, for these choices of $\beta$ and $k$ every $\beta$-quasi-parabolic symbol is a parabolic $k$-quasi-radial symbol, so we have
    \[
        L^\infty(D_n)_{\beta,P(n)} \subset L^\infty(D_n)_{\beta(k),P(n)}.
    \]
    However, this inclusion is proper. To see this consider the symbol $a : D_n \rightarrow \C$ defined by
    \[
        a_2(z) = |z_{(2)}|.
    \]
    Clearly, $a_2$ belongs to $L^\infty(D_n)_{\beta(k),P(n)}$. If we assume the existence of a function $f : [0,\infty)^2 \rightarrow \C$ such that
    \[
        |z_{(2)}| = a_2(z) = f(|z_1|, \im(z_n) - |z'|^2)
    \]
    for all $z \in D_n$, then by taking $z$ such that $z_{(2)} = 0$ we conclude
    \[
        0 = f(|z_1|, \im(z_n)-|z_1|^2)
    \]
    for every $(z_1,z_n) \in D_2$. This clearly implies $f \equiv 0$, which is a contradiction. Hence we have
    \[
        a_2 \in L^\infty(D_n)_{\beta(k),P(n)} \setminus
                L^\infty(D_n)_{\beta,P(n)},
    \]
    thus proving the proper inclusion
    \[
        L^\infty(D_n)_{\beta,P(n)} \subsetneq L^\infty(D_n)_{\beta(k),P(n)}.
    \]

    It is a simple exercise to adapt this argument to show that for any $k \in \Z^m_+$ so that $k = (k',1)$ and with $m \geq 3$ we have
    \[
        L^\infty(D_n)_{\beta,P(n)} \neq L^\infty(D_n)_{\beta(k),P(n)}.
    \]
    Finally, for the partition $k = (n-1,1)$, if the symbol given by $z \mapsto |z'|$ belongs to $L^\infty(D_n)_{\beta,P(n)}$, then the symbol $z \mapsto |(0, z_2, \dots, z_{n-1}, 0)|$ is an element of $L^\infty(D_n)_{\beta,P(n)}$ as well, which contradicts the previous arguments. This completes the proof that
    \[
        L^\infty(D_n)_{\beta,P(n)} \neq L^\infty(D_n)_{\beta(k),P(n)}
    \]
    for every partition $k$ of $n$ such that $k = (k',1)$.
\end{proof}

\subsection{$\beta$-quasi-hyperbolic symbols}
Let us now consider the group $G = \T^{n-1}\times \R_+$ acting by the quasi-hyperbolic action $H(n)$ on $D_n$ defined in Section~\ref{sec:MASG_momentmaps}. A remark as the one given at the beginning of Subsection~\ref{subsec:beta-quasi-parabolic} can be applied in this case. In other words, for a linearly independent set $\beta = \{v_1, \dots, v_m\} \subset \R^n$ taken as a basis of the Lie algebra $\fh$ of the connected subgroup $H$ of $\T^{n-1}\times \R_+$, we can assume that the last column of $A(\beta)$ is non-zero.

We now describe the moment maps for connected subgroups of $\T^{n-1}\times\R_+$. This result is a consequence of Propositions~\ref{prop:momentmap_H(n)} and \ref{prop:momentmap_AbelianfromMASG-basis} and Corollary~\ref{cor:momentmapfunctions_betaG}.

\begin{proposition}\label{prop:momentmap_betaH(n)}
    Let $H$ be a connected Abelian subgroup of $\T^{n-1}\times\R_+$ whose Lie algebra $\fh$ has the basis $\beta = \{ v_1, \dots, v_m \}$. Then, the following hold.
    \begin{enumerate}
        \item If $\beta$ is orthogonal, then the moment map for the $H$-action on $D_n$ is the map $\mu^H : D_n \rightarrow \fh$ given by
        \[
            \mu(z) = -\frac{1}{2(\im(z_n)-|z'|^2)} \sum_{j=1}^m
                \frac{\langle(2|z_1|^2, \dots, 2|z_{n-1}|^2, \re(z_n)),v_j\rangle}{\langle v_j,v_j\rangle} v_j,
        \]
        for every $z \in D_n$.
        \item If $\beta$ is an arbitrary basis, then the space of essentially bounded moment map functions for $H$ is given by
        \[
            L^\infty(D_n)^{\mu^H} = \{ f(a_1, \dots, a_m) \mid
                f \in L^\infty((a_1, \dots, a_m)(D_n)) \},
        \]
        where $a_j : D_n \rightarrow \R$ are defined by
        \[
            a_j(z) = \frac{1}{2(\im(z_n)-|z'|^2)}
                \langle(2|z_1|^2, \dots, 2|z_{n-1}|^2, \re(z_n)),v_j\rangle,
        \]
        for every $z \in D_n$ and for every $j = 1, \dots, m$.
    \end{enumerate}

\end{proposition}

In this case, we obtain from the previous result the following special symbols.

\begin{definition}\label{def:beta-quasi-hyperbolic}
    Let $\beta = \{ v_1, \dots, v_m \} \subset \R^n$ be a linearly independent set. An essentially bounded function $a : D_n \rightarrow \C$ is called a $\beta$-quasi-hyperbolic symbol if there is a measurable function $f$ such that
    \[
        a = f(a_1, \dots, a_m),
    \]
    where $a_1, \dots, a_m$ are given as in Proposition~\ref{prop:momentmap_betaH(n)}. The space of all essentially bounded $\beta$-quasi-hyperbolic symbols will be denoted by $L^\infty(D_n)_{\beta,H(n)}$. In other words, we have
    \[
        L^\infty(D_n)_{\beta,H(n)} = L^\infty(D_n)^{\mu^H},
    \]
    where $H = \exp(\R\langle\beta\rangle)$.
\end{definition}

As noted in Remark~\ref{rmk:momentmapfunctions_beta_matrices}, for $H = \T^{n-1}\times\R_+$ and for any basis $\beta$ of $\R^n$, the $\beta$-quasi-hyperbolic symbols are precisely the quasi-hyperbolic symbols from \cite{QVUnitBall1}. Hence, we now obtain the following generalization of the commutativity of Toeplitz operators with quasi-hyperbolic symbols stated in \cite{QVUnitBall1}. This result is a consequence of Theorem~\ref{thm:Toeplitz_mu-functions}.

\begin{theorem}\label{thm:commToeplitz-beta-quasi-hyperbolic}
    Let $\beta = \{ v_1, \dots, v_m\} \subset \R^n$ be a linearly independent set. Then, for every $\lambda > -1$ the $C^*$-algebra $\cT^{(\lambda)}(L^\infty(D_n)_{\beta,H(n)})$ generated by Toeplitz operators with essentially bounded $\beta$-quasi-hyperbolic symbols acting on $\cA^2_\lambda(D_n)$ is commutative.
\end{theorem}

\subsection{$\beta$-nilpotent symbols}
Consider the group $G = \R^n$ acting by the nilpotent action $N(n)$ on $D_n$ defined in Section~\ref{sec:MASG_momentmaps}. The moment maps for connected subgroups of $\R^n$ is given in the next result. It is a consequence of Propositions~\ref{prop:momentmap_N(n)} and \ref{prop:momentmap_AbelianfromMASG-basis} and Corollary~\ref{cor:momentmapfunctions_betaG}.

\begin{proposition}\label{prop:momentmap_betaN(n)}
    Let $H$ be a connected Abelian subgroup of $\R^n$ whose Lie algebra $\fh$ has the basis $\beta = \{ v_1, \dots, v_m \}$. Then, the following hold.
    \begin{enumerate}
        \item If $\beta$ is orthogonal, then the moment map for the $H$-action on $D_n$ is the map $\mu^H : D_n \rightarrow \fh$ given by
        \[
            \mu(z) = -\frac{1}{2(\im(z_n)-|z'|^2)} \sum_{j=1}^m
                \frac{\langle (-4\im(z'),1),v_j\rangle}{\langle v_j,v_j\rangle} v_j,
        \]
        for every $z \in D_n$.
    \item If $\beta$ is an arbitrary basis, then the space of essentially bounded moment map functions for $H$ is given by
        \[
            L^\infty(D_n)^{\mu^H} = \{ f(a_1, \dots, a_m) \mid
                f \in L^\infty((a_1, \dots, a_m)(D_n)) \},
        \]
        where $a_j : D_n \rightarrow \R$ are defined by
        \[
            a_j(z) = \frac{1}{2(\im(z_n)-|z'|^2)}
                \langle(-4\im(z'),1),v_j\rangle,
        \]
        for every $z \in D_n$ and for every $j = 1, \dots, m$.
    \end{enumerate}
\end{proposition}

In this case, we obtain the following special symbols.

\begin{definition}\label{def:beta-nilpotent}
    Let $\beta = \{ v_1, \dots, v_m \} \subset \R^n$ be a linearly independent set. An essentially bounded function $a : D_n \rightarrow \C$ is called a $\beta$-nilpotent symbol if there is a measurable function $f$ such that
    \[
        a = f(a_1, \dots, a_m),
    \]
    where $a_1, \dots, a_m$ are given as in Proposition~\ref{prop:momentmap_betaN(n)}. The space of all essentially bounded $\beta$-nilpotent symbols will be denoted by $L^\infty(D_n)_{\beta,N(n)}$. In other words, we have
    \[
        L^\infty(D_n)_{\beta,N(n)} = L^\infty(D_n)^{\mu^H},
    \]
    where $H = \exp(\R\langle\beta\rangle)$.
\end{definition}

\begin{example}[Some nilpotent symbols (\cite{SVNilpotentDuduchava2017}) as $\beta$-nilpotent symbols]
    \label{ex:beta-nilpotent-noradial}
    As it is well known (see \cite{QVUnitBall1}) the nilpotent symbols are functions of the following expression
    \[
        (\im(z'), \im(z_n) - |z'|^2).
    \]
    These are at the same time our $\beta$-nilpotent symbols for any basis $\beta$ of $\R^n$. On the other hand, in Section~3 of \cite{SVNilpotentDuduchava2017} the authors considered symbols that are functions of the expression
    \[
        (\im(z_{k+1}), \dots, \im(z_{n-1}), \im(z_n)-|z'|^2),
    \]
    for some $k$ such that $1 \leq k \leq n-1$. It is easy to see from the definitions that these are precisely the $\beta$-nilpotent symbols for the linearly independent subset of $\R^n$ given by the rows of the following $(n-k)\times n$ matrix
    \[
        A(\beta) = (0, I_{n-k}),
    \]
    where we have used the notation from Remark~\ref{rmk:momentmapfunctions_beta_matrices}. Hence, such special type of symbols considered in \cite{SVNilpotentDuduchava2017} is a particular case of our $\beta$-nilpotent symbols.
\end{example}

We now obtain the following result that generalizes the commutativity of Toeplitz operators with nilpotent symbols stated in \cite{QVUnitBall1}, as well as some of the commutativity results from \cite{SVNilpotentDuduchava2017}. It is a consequence of Theorem~\ref{thm:Toeplitz_mu-functions}.

\begin{theorem}\label{thm:commToeplitz-beta-nilpotent}
    Let $\beta = \{ v_1, \dots, v_m\} \subset \R^n$ be a linearly independent set. Then, for every $\lambda > -1$ the $C^*$-algebra $\cT^{(\lambda)}(L^\infty(D_n)_{\beta,N(n)})$ generated by Toeplitz operators with essentially bounded $\beta$-nilpotent symbols acting on $\cA^2_\lambda(D_n)$ is commutative.
\end{theorem}

\subsection{$\beta$-quasi-nilpotent symbols}
Finally, let us consider for $1 \leq k \leq n-1$ the group $G = \T^k\times\R^{n-k}$ acting by the quasi-nilpotent action $N(n,k)$ on $D_n$ defined in Section~\ref{sec:MASG_momentmaps}. As in the previous cases, there are some restrictions we may assume to avoid repetitions with other actions. As before, we have a linearly independent set $\beta = \{ v_1, \dots, v_m \} \subset \R^n$, the Lie algebra $\fh = \R\langle\beta\rangle$ and the connected Abelian subgroup $H = \exp(\fh)$. Following Remark~\ref{rmk:momentmapfunctions_beta_matrices}, there is a corresponding matrix $A(\beta)$. If the first $k$ columns of $A(\beta)$ vanish, then $H$ is a subgroup of the nilpotent action. If the last $n-k$ columns of $A(\beta)$ vanish, then $H$ is conjugate to a subgroup of the quasi-elliptic action. If the columns with indices $k+1$ to $n-1$ vanish, then $H$ is a subgroup of the quasi-parabolic action. Hence, as long as the other actions have been considered, we may assume that $A(\beta)$ has some non-zero column on each of the three submatrices involved in such discussion. By doing so, we consider the most general subgroup $H$ of the quasi-nilpotent case.

The moment maps for connected subgroups in this case are given by the next result. It is a consequence of Propositions~\ref{prop:momentmap_N(n,k)} and \ref{prop:momentmap_AbelianfromMASG-basis} and Corollary~\ref{cor:momentmapfunctions_betaG}.

\begin{proposition}\label{prop:momentmap_betaN(n,k)}
    Let $H$ be a connected Abelian subgroup of $\T^k\times\R^{n-k}$ whose Lie algebra $\fh$ has the basis $\beta = \{ v_1, \dots, v_m \}$. Then, the following hold.
    \begin{enumerate}
        \item If $\beta$ is orthogonal, then the moment map for the $H$-action on $D_n$ is the map $\mu^H : D_n \rightarrow \fh$ given by
            \[
                \mu(z) = -\frac{1}{2(\im(z_n)-|z'|^2)} \sum_{j=1}^m
                    \frac{\langle (2|z_1|^2, \dots, 2|z_k|^2,-4\im(z_{(2)}),1),v_j\rangle}{\langle v_j,v_j\rangle} v_j,
            \]
            for every $z \in D_n$.
        \item If $\beta$ is an arbitrary basis, then the space of essentially bounded moment map functions for $H$ is given by
            \[
                L^\infty(D_n)^{\mu^H} = \{ f(a_1, \dots, a_m) \mid
                    f \in L^\infty((a_1, \dots, a_m)(D_n)) \},
            \]
            where $a_j : D_n \rightarrow \R$ are defined by
            \[
                a_j(z) = \frac{1}{2(\im(z_n)-|z'|^2)}
                    \langle (2|z_1|^2, \dots 2|z_k|^2, -4\im(z_{(2)}), 1), v_j\rangle,
            \]
            for all $z \in D_n$ and for every $j = 1, \dots, m$.
    \end{enumerate}
\end{proposition}

Now we obtain the following special symbols.

\begin{definition}\label{def:beta-quasi-nilpotent}
    Let $\beta = \{ v_1, \dots, v_m \} \subset \R^n$ be a linearly independent set. An essentially bounded function $a : D_n \rightarrow \C$ is called a $\beta$-quasi-nilpotent symbol if there is a measurable function $f$ such that
    \[
        a = f(a_1, \dots, a_m),
    \]
    where $a_1, \dots, a_m$ are given as in Proposition~\ref{prop:momentmap_betaN(n,k)}. The space of all essentially bounded $\beta$-nilpotent symbols will be denoted by $L^\infty(D_n)_{\beta,N(n,k)}$. In other words, we have
    \[
        L^\infty(D_n)_{\beta,N(n,k)} = L^\infty(D_n)^{\mu^H},
    \]
    where $H = \exp(\R\langle\beta\rangle)$.
\end{definition}

\begin{example}[$\alpha$-quasi-nilpotent quasi-radial symbols as $\beta$-quasi-nilpotent symbols]
\label{ex:beta-quasi-nilpotent_partitions}
    As it corresponds to this subsection we have $k \in \Z_+$ that defines the action of $\T^k\times\R^{n-k}$. We will consider the most general case of the quasi-nilpotent action as described at the beginning of this subsection. Hence, we will assume that $1 \leq k \leq n-2$. Let us also consider a partition $\alpha \in \Z_+^m$ of $k$. We complete this to a partition $\widehat{\alpha}$ of $n$ by adding $n-k$ entries with $1$ on each new entry as follows
    \[
        \widehat{\alpha} = (\alpha,1, \dots, 1) \in \Z_+^{m+n-k}.
    \]
    Using the notation of Examples~\ref{ex:quasi-radial_partitions} and \ref{ex:beta-quasi-parabolic_partitions} we define the linearly independent set $\beta(\widehat{\alpha}) = \{v_1, \dots, v_{m+n-k}\}$ by choosing
    \[
        v_j = (0, 1_{\widehat{\alpha}_j}, 0)
    \]
    for every $j = 1, \dots, m+n-k$, where the non-zero entries occur exactly at the indices $\widehat{\alpha}_0 + \dots + \widehat{\alpha}_{j-1}+1, \dots, \widehat{\alpha}_0 + \dots + \widehat{\alpha}_j$. In the notation of Remark~\ref{rmk:momentmapfunctions_beta_matrices}, the matrix associated to $\beta(\widehat{\alpha})$ is given by
    \[
        A(\beta(\widehat{\alpha})) =
        \begin{pmatrix}
          1_{\widehat{\alpha}_1} & 0 & \cdots & 0 & 0 \\
          0 & 1_{\widehat{\alpha}_2} & \cdots & 0 & 0  \\
          \vdots & \vdots & \vdots & \vdots & \vdots \\
          0 & 0 & \cdots & 1_{\widehat{\alpha}_m} & 0 \\
          0 & 0 & \cdots & 0 & I_{n-k} \\
        \end{pmatrix}.
    \]
    For the basis $\beta(\widehat{\alpha})$ just defined, the functions $a_j$ from Proposition~\ref{prop:momentmap_betaN(n,k)} are given by
    \[
        a_j(z) =
        \begin{cases}
            \displaystyle
            \frac{|z_{(\widehat{\alpha}_j)}|^2}{\im(z_n)-|z'|^2},
                    & \mbox{if } j=1, \dots, m, \\
            \displaystyle
            -\frac{2\im(z_j)}{\im(z_n)-|z'|^2}
                    & \mbox{if } j=m+1, \dots, m+n-k-1 \\
            \displaystyle
            \frac{1}{2(\im(z_n)-|z'|^2)}, & \mbox{if } j = m+n-k,
        \end{cases}
    \]
    for all $z \in D_n$. Using a simple change of variables, it follows that the $\beta(\widehat{\alpha})$-quasi-nilpotent symbols are exactly the functions of the expression
    \[
        (|z_{(1)}|, \dots, |z_{(m-1)}|, \im(z_{k+1}), \dots, \im(z_{n-1}),
        \im(z_n)-|z''|^2),
    \]
    where $z'' = (z_{k+1}, \dots, z_n)$.
    It follows that, for this particular choice of $\beta(\widehat{\alpha})$, the $\beta(\widehat{\alpha})$-quasi-parabolic symbols from Definition~\ref{def:beta-quasi-parabolic} are precisely the $\alpha$-quasi-nilpotent quasi-radial symbols from \cite{BauerVasilevskiQuasiNilpotent2012}.
\end{example}

By the previous example the next result generalizes the commutativity of Toeplitz operators with quasi-nilpotent symbols stated in \cite{QVUnitBall1} and with $\alpha$-quasi-nilpotent quasi-radial symbols stated in \cite{BauerVasilevskiQuasiNilpotent2012}. This result is a consequence of Theorem~\ref{thm:Toeplitz_mu-functions}.

\begin{theorem}\label{thm:commToeplitz-beta-quasi-nilpotent}
    Let $\beta = \{ v_1, \dots, v_m\} \subset \R^n$ be a linearly independent set. Then, for every $\lambda > -1$ the $C^*$-algebra $\cT^{(\lambda)}(L^\infty(D_n)_{\beta,N(n,k)})$ generated by Toeplitz operators with essentially bounded $\beta$-quasi-nilpotent symbols acting on $\cA^2_\lambda(D_n)$ is commutative.
\end{theorem}

As in the $\beta$-quasi-elliptic and $\beta$-quasi-parabolic cases, it is possible to prove the existence of a $\beta$ such that the $\beta$-quasi-nilpotent symbols are not given as a set of $\alpha$-quasi-nilpotent quasi-radial symbols from \cite{BauerVasilevskiQuasiNilpotent2012}. Similar arguments can be used to achieve this. In fact, the $\alpha$-quasi-nilpotent quasi-radial symbols from \cite{BauerVasilevskiQuasiNilpotent2012} do not consider any sort of partition in the coordinates from the index $k+1$ to the index $n$. We thus have the following result.

\begin{proposition}\label{prop:betaquasinilpotent-notquasinilpotentquasi-radial}
    Assume that $n \geq 4$. Then, there exists $\beta \subset \R^n$ linearly independent such that the space of $\beta$-quasi-nilpotent symbols are different from the space of $\alpha$-quasi-nilpotent quasi-radial symbols defined by any partition $\alpha$ of $k$. In other words, we have
    \[
        L^\infty(D_n)_{\beta,N(n,k)} \neq L^\infty(D_n)_{\beta(\widehat{\alpha}),N(n,k)}
    \]
    for every partition $\alpha$ of $k$, where $\beta(\widehat{\alpha})$ is defined in Example~\ref{ex:beta-quasi-nilpotent_partitions}.
\end{proposition}

\section{Spectral integral formulas for Toeplitz operators with $\beta$-symbols}\label{sec:spectralformulas}
As before, let us consider a MASG $G$ of the group of biholomorphisms of $D$, which is either $\B^n$ or $D_n$. Next we can choose a connected Abelian subgroup $H$ of $G$. As noted in Remark~\ref{rmk:Hmomentfunctions_Gmomentfunctions}, and as a consequence of Propositions~\ref{prop:H-invariance_momentmaps} and \ref{prop:momentmap_AbelianfromMASG}, every moment map function for $H$ is a $G$-invariant symbol. In particular, the spectral integral formulas from \cite{QVUnitBall1} for the Toeplitz operators given for the five types of MASGs can be applied when the symbols are moment map functions for $H$. This will yield corresponding spectral integral formulas for all the $\beta$-symbols introduced in Section~\ref{sec:commToep_momentmaps_Abelian}. Then, by applying a change of variable associated to the moment map we will obtain spectral integral formulas for the case of $G$-invariant symbols that will provide an important simplification of the known expressions. Finally, we also obtain corresponding simplified formulas in the moment map coordinates for the spectral integral formulas for all $\beta$-symbols.

In all subsections below, $\beta = \{v_1, \dots, v_m\} \subset \R^n$ will denote a linearly independent set. This set $\beta$ together with the choice of the MASG $G$ determine the specific type of the $\beta$-symbols considered.

\subsection{Spectra for $\beta$-quasi-elliptic symbols}
By Proposition~\ref{prop:momentmap_betaE(n)} we have that a symbol $a : \B^n \rightarrow \C$ is $\beta$-quasi-elliptic if there is function $f$ such that
\[
    a(z) = f\bigg(
        \frac{\langle r^2, v_1\rangle}{1-|r^2|}, \dots,
        \frac{\langle r^2, v_m\rangle}{1-|r^2|}
    \bigg).
\]
For this expression we follow some special notation. In the first place, we will denote
\begin{align*}
    r &= (r_1, \dots, r_n) = (|z_1|, \dots, |z_n|), \\
    r^2 &= (r_1^2, \dots, r_n^2) = (|z_1|^2, \dots, |z_n|^2).
\end{align*}
On the other hand, and as usual, for every $w \in \C^k$ the expression $|w|$ will denote its Euclidean norm. However, for the particular case of $x \in \R^n_+$ we will use
\[
    |x| = x_1 + \dots + x_n
\]
and this notation will stand for the rest of this work. In particular, we have
\[
    |r^2| = r_1^2 + \dots + r_n^2.
\]
Hence, any $\beta$-quasi-elliptic symbol is a function of $r$ and so it is a quasi-elliptic symbol as defined in \cite{QVUnitBall1}. In particular, we can apply Theorem~10.1 from \cite{QVUnitBall1} to obtain the following result.

\begin{theorem}\label{thm:specToeplitz-beta-quasi-elliptic}
    Suppose that $G = \T^n$ which acts on $\B^n$ by the action $E(n)$ and let $\beta = \{v_1, \dots, v_m\} \subset \R^n$ be a linearly independent set. Then, for every $\lambda > -1$ there exists a unitary map $R : \cA^2_\lambda(\B^n) \rightarrow \ell^2(\N^n)$ such that for every essentially bounded $\beta$-quasi-elliptic symbol $a : \B^n \rightarrow \C$ of the form
    \[
        a(z) = f\bigg(
            \frac{\langle r^2, v_1\rangle}{1-|r^2|}, \dots,
            \frac{\langle r^2, v_m\rangle}{1-|r^2|}
        \bigg)
    \]
    we have $R T_a^{(\lambda)} R^* = \gamma_{a,\lambda} I$, a multiplication operator, where $\gamma_{a,\lambda}$ is given by
    \begin{multline*}
        \gamma_{a,\lambda}(p) = \\
        =\frac{2^n\Gamma(n+|p|+\lambda+1)}{p!\Gamma(\lambda +1)}
        \int_{\tau(\B^n)}
        f\bigg(
            \frac{\langle r^2, v_1\rangle}{1-|r^2|}, \dots,
            \frac{\langle r^2, v_m\rangle}{1-|r^2|}
        \bigg) r^{2p}(1-|r^2|)^\lambda \prod_{j=1}^n r_j \dif r_j,
    \end{multline*}
    for every $p \in \N^n$.
\end{theorem}

We now write the spectral integral formulas of Theorem~\ref{thm:specToeplitz-beta-quasi-elliptic} in terms of moment map coordinates for the action $E(n)$ in the case of quasi-elliptic symbols.

\begin{theorem}\label{thm:specToeplitz-quasi-elliptic}
    For every $\lambda > -1$ there exists a unitary map $R : \cA^2_\lambda(\B^n) \rightarrow \ell^2(\N^n)$ such that for every essentially bounded quasi-elliptic symbol $a : \B^n \rightarrow \C$ of the form
    \[
        a(z) = f\bigg(
            \frac{r_1^2}{1-|r^2|}, \dots, \frac{r_n^2}{1-|r^2|}
            \bigg)
    \]
    we have $R T_a^{(\lambda)} R^* = \gamma_{a,\lambda} I$, a multiplication operator, where $\gamma_{a,\lambda}$ is given by
    \[
        \gamma_{a,\lambda}(p) =
            \frac{\Gamma(\lambda+|p|+n+1)}{p!\Gamma(\lambda + 1)}
                \int_{\mathbb{R}^{n}_{+}}
                \frac{f(u) u^{p}}{(1+|u|)^{\lambda+|p|+n+1} } \dif u,
    \]
    for every $p \in \N^n$.
\end{theorem}
\begin{proof}
    We use a change of coordinates corresponding to the moment map of $E(n)$. More precisely, we introduce the map $\varphi : \tau(\B^n) \rightarrow \R^n_+$ defined by
    \[
        u = \varphi(r)
            = \bigg( \frac{r_1^2}{1-|r^2|},\dots,
            \frac{r_n^2}{1-|r^2|} \bigg),
    \]
    which has an inverse given by
    \[
        \varphi^{-1}(u) =
            \Bigg(  \bigg(\frac{u_1}{1+|u|}  \bigg)^{1/2}, \dots,
            \bigg(\frac{u_n}{1+|u|} \bigg)^{1/2} \Bigg).
    \]
    For this change of coordinates we have
    \[
        1-|r^2| = \frac{1}{1+|u|}, \quad
        \prod_{j=1}^n 2 r_j \dif r_j  = \frac{\dif u}{(1+|u|)^{n+1}}.
    \]
    Hence, the integral formula is obtained from Theorem~\ref{thm:specToeplitz-beta-quasi-elliptic} by taking $\beta = \{e_1, \dots, e_n\}$, the canonical basis of $\R^n$, and applying the change of coordinates $\varphi$.
\end{proof}

The next result yields a corresponding spectral integral formulas for $\beta$-quasi-elliptic symbols. We recall from Remark~\ref{rmk:momentmapfunctions_beta_matrices} that $A(\beta)$ is the matrix whose rows are the elements $v_1, \dots, v_m$ of $\beta$. We also recall from that same remark our abuse of notation where for a function $f$ of several variables we use interchangeably a column or a row for the arguments of $f$.

\begin{corollary}\label{cor:specToeplitz-beta-quasi-elliptic}
    Suppose that $G = \T^n$ which acts on $\B^n$ by the action $E(n)$ and let $\beta = \{v_1, \dots, v_m\} \subset \R^n$ be a linearly independent set. Then, for every $\lambda > -1$ there exists a unitary map $R : \cA^2_\lambda(\B^n) \rightarrow \ell^2(\N^n)$ such that for every essentially bounded $\beta$-quasi-elliptic symbol $a : \B^n \rightarrow \C$ of the form
    \[
        a(z) = f\bigg(
            \frac{\langle r^2, v_1\rangle}{1-|r|^2}, \dots,
            \frac{\langle r^2, v_m\rangle}{1-|r|^2}
        \bigg)
    \]
    we have $R T_a^{(\lambda)} R^* = \gamma_{a,\lambda} I$, a multiplication operator, where $\gamma_{a,\lambda}$ is given by
     \[
        \gamma_{a,\lambda}(p) =
            \frac{\Gamma(\lambda+|p|+n+1)}{p!\Gamma(\lambda +1)}
                \int_{\R^n_+}
            \frac{f(A(\beta) u^\top) u^{p}}{(1+|u|)^{\lambda+|p|+n+1} } \dif u,
    \]
    for every $p \in \N^n$.
\end{corollary}
\begin{proof}
    We observe that the symbol $a$ in the statement can be rewritten as
    \[
        a(z) = f\Bigg(
            A(\beta)
            \bigg(
            \frac{r_1^2}{1-|r^2|}, \dots, \frac{r_n^2}{1-|r^2|}
            \bigg)^\top
            \Bigg),
    \]
    Hence, the result is a consequence of Theorem~\ref{thm:specToeplitz-beta-quasi-elliptic}.
\end{proof}

\subsection{Spectra for $\beta$-quasi-parabolic symbols}
We now apply Proposition~\ref{prop:momentmap_betaP(n)} to see that a symbol $a : D_n \rightarrow \C$ is $\beta$-quasi-parabolic if there is a function $f$ such that
\[
    a(z) = f\bigg(
    \frac{\langle(2r'^2,1),v_1\rangle}{2(\im(z_n)-|r'^2|)}, \dots,
    \frac{\langle(2r'^2,1),v_m\rangle}{2(\im(z_n)-|r'^2|)}
    \bigg)
\]
where we denote
\begin{align*}
    r' &= (r_1, \dots, r_{n-1}) = (|z_1|, \dots, |z_{n-1}|), \\
    r'^2 &= (r_1^2, \dots, r_{n-1}^2) = (|z_1|^2, \dots, |z_{n-1}|^2),\\
\end{align*}
from which we have
\[
    |r'^2| = r_1^2 + \dots + r_{n-1}^2.
\]
In particular, every $\beta$-quasi-parabolic symbol is a function of $(r',\im(z_n))$ and so it is quasi-parabolic in the sense of \cite{QVUnitBall1}. An application of Theorem~10.2 from \cite{QVUnitBall1} yields the following result.

\begin{theorem}\label{thm:specToeplitz-beta-quasi-parabolic}
    Suppose that $G = \T^{n-1}\times\R$ which acts on $D_n$ by the action $P(n)$ and let $\beta = \{v_1, \dots, v_m\} \subset \R^n$ be a linearly independent set. Then, for every $\lambda > -1$ there exists a unitary map $R : \cA^2_\lambda(D_n) \rightarrow \ell^2(\N^{n-1},L^2(\R_+))$ such that for every essentially bounded $\beta$-quasi-elliptic symbol $a : D_n \rightarrow \C$ of the form
    \[
        a(z) = f\bigg(
        \frac{\langle(2r'^2,1),v_1\rangle}{2(\im(z_n)-|r'|^2)}, \dots,
        \frac{\langle(2r'^2,1),v_m\rangle}{2(\im(z_n)-|r'|^2)}
        \bigg)
    \]
    we have $RT_a^{(\lambda)} R^* = \gamma_{a,\lambda} I$, a multiplication operator, where $\gamma_{a,\lambda}$ is given by
    \[
        \gamma_{a,\lambda}(p,\xi)=
            \frac{(2\xi)^{\lambda+|p|+n}}{p!\Gamma(\lambda+1)}
            \int_{\R^n_+}
            f\bigg(
            \frac{\langle(2r',1),v_1\rangle}{2r_n}, \dots,
            \frac{\langle(2r',1),v_m\rangle}{2r_n}
            \bigg) r'^p r_n^\lambda e^{-2\xi|r|} \dif r,
    \]
    for every $p \in \N^{n-1}$ and $\xi \in \R_+$.
\end{theorem}

We now obtain spectral integral formulas corresponding to Theorem~\ref{thm:specToeplitz-beta-quasi-parabolic} in terms of the moment map coordinates for the action $P(n)$ for the quasi-parabolic symbols.

\begin{theorem}\label{thm:specToeplitz-quasi-parabolic}
    For every weight $\lambda > -1$, there exists a unitary transformation $R : \cA^2_\lambda(D_n) \rightarrow \ell^2(\N^{n-1},L^2(\R_+))$ such that for every essentially bounded quasi-parabolic symbol $a : D_n \rightarrow \C$ of the form
    \[
        a(z) = f\bigg(
                \frac{r_1^2}{\im(z_n)-|r'^2|}, \dots,
                \frac{r_{n-1}^2}{\im(z_n)-|r'^2|},
                \frac{1}{2(\im(z_n)-|r'|^2)}
            \bigg)
    \]
    we have $RT_a^{(\lambda)} R^* = \gamma_{a,\lambda} I$, a multiplication operator, where $\gamma_{a,\lambda}$ is given by
    \[
        \gamma_{a,\lambda}(p,\xi)=
            \frac{\xi^{\lambda+|p|+n}}{p!\Gamma(\lambda+1)}
            \int_{\R^n_+}
            \frac{f(u) u'^p e^{-\frac{\xi}{u_{n}}(|u'|+1)} \dif u}{ u_n^{\lambda+|p|+n+1}},
    \]
    for every $p \in \N^{n-1}$ and $\xi \in \R_+$.
\end{theorem}
\begin{proof}
    Let us consider a change of coordinates related to the moment map of $P(n)$. In other words, we consider the function $\varphi : \R^n_+ \rightarrow \R^n_+$ given by
    \[
        u = \varphi(r) =
        \bigg(
            \frac{r'}{r_n}, \frac{1}{2r_n}
        \bigg),
    \]
    whose inverse is given by
    \[
        \varphi^{-1}(u) = \frac{1}{2u_n} (u',1).
    \]
    For this change of coordinates we have
    \[
        \dif r = \frac{\dif u}{2^n u_n^{n+1}}.
    \]
    Hence, the integral formula is obtained from Theorem~\ref{thm:specToeplitz-beta-quasi-parabolic} by taking the canonical basis $\beta = \{ e_1, \dots, e_n\}$ of $\R^n$ and applying the change of coordinates $\varphi$.
\end{proof}

The next result now provides the corresponding spectral integral formulas for $\beta$-quasi-parabolic symbols. As before, we follow the conventions and notation from Remark~\ref{rmk:momentmapfunctions_beta_matrices}.

\begin{corollary}\label{cor:specToeplitz-beta-quasi-parabolic}
    Suppose that $G = \T^{n-1}\times\R$ which acts on $D_n$ by the action $P(n)$ and let $\beta = \{v_1, \dots, v_m\} \subset \R^n$ be a linearly independent set. Then, for every $\lambda > -1$ there exists a unitary map $R : \cA^2_\lambda(D_n) \rightarrow \ell^2(\N^{n-1},L^2(\R_+))$ such that for every essentially bounded $\beta$-quasi-parabolic symbol $a : D_n \rightarrow \C$ of the form
    \[
        a(z) = f\bigg(
            \frac{\langle (2r'^2,1), v_1\rangle}{2(\im(z_n)-|r'|^2)}, \dots,
            \frac{\langle (2r'^2,1), v_m\rangle}{2(\im(z_n)-|r'|^2)}
        \bigg)
     \]
    we have $RT_a^{(\lambda)} R^* = \gamma_{a,\lambda} I$, a multiplication operator, where $\gamma_{a,\lambda}$ is given by
    \[
        \gamma_{a,\lambda}(p,\xi)=
            \frac{\xi^{\lambda+|p|+n}}{p!\Gamma(\lambda+1)}
            \int_{\R^n_+}
            \frac{f(A(\beta)u^\top) u'^p e^{-\frac{\xi}{u_{n}}(|u'|+1)} \dif u}{ u_n^{\lambda+|p|+n+1}},
    \]
    for every $p \in \N^{n-1}$ and $\xi \in \R_+$.
\end{corollary}
\begin{proof}
    The $\beta$-quasi-parabolic symbols considered can be rewritten as
    \[
        a(z) =
        f\bigg(
            A(\beta)
                \frac{(2r'^2,1)^\top}{2(\im(z_n)-|r'^2|)}
        \bigg),
    \]
    and so we can apply Theorem~\ref{thm:specToeplitz-beta-quasi-parabolic} to obtain the desired result.
\end{proof}

\subsection{Spectra for $\beta$-quasi-hyperbolic symbols.}
\label{subsec:beta-quasi-hyperbolic-spectra}
As a consequence of Proposition~\ref{prop:momentmap_betaH(n)} a symbol $a : D_n \rightarrow \C$ is $\beta$-quasi-hyperbolic if there is a function $f$ such that
\[
    a(z) =
    f\bigg(
        \frac{\langle(2r'^2,\re(z_n)),v_1\rangle}{2(\im(z_n)-|z'|^2)}, \dots,
        \frac{\langle(2r'^2,\re(z_n)),v_m\rangle}{2(\im(z_n)-|z'|^2)}
    \bigg),
\]
where we have used the notation from the $\beta$-quasi-parabolic case. In particular, every $\beta$-quasi-hyperbolic symbol is a function of the expression
\[
    \bigg(
    \frac{|z_1|^2}{\im(z_n)-|z'|^2}, \dots, \frac{|z_{n-1}|^2}{\im(z_n)-|z'|^2},
    \frac{\re(z_n)}{\im(z_n)-|z'|^2}
    \bigg).
\]
On the other hand, a quasi-hyperbolic symbol as defined in \cite{QVUnitBall1} is a function of the expression
\begin{multline*}
    (f_1(z), \dots, f_{n-1}(z), f_n(z)) = \\
        \bigg(
            \frac{|z_1|}{\sqrt{|z'|^2 + |z_n-i|z'|^2|}}, \dots,
            \frac{|z_{n-1}|}{\sqrt{|z'|^2 + |z_n-i|z'|^2|}},
            \arg(z_n-i|z'|^2)
        \bigg).
\end{multline*}
After some simple coordinate manipulations we obtain the following identities
\begin{align*}
    \frac{\re(z_n)}{\im(z_n) - |z'|^2} &=
        \cot(f_n(z)) \\
    \frac{|z_j|^2}{\im(z_n) - |z'|^2} &=
        \frac{f_j(z)^2 \csc(f_n(z))}{1-\sum_{k=1}^{n-1} f_k(z)^2}
\end{align*}
for all $j=1,\dots,n-1$. These prove that every $\beta$-quasi-hyperbolic symbol is quasi-hyperbolic in the sense of \cite{QVUnitBall1}. Hence, we can apply Theorem~10.5 from \cite{QVUnitBall1} to obtain the representation of a Toeplitz operator as a multiplication operator when the symbol is $\beta$-quasi-hyperbolic. For this particular case we omit the explicit formula which is already quite complicated even in the quasi-hyperbolic case.

\subsection{Spectra for $\beta$-nilpotent symbols.} By Proposition~\ref{prop:momentmap_betaN(n)} any given symbol $a : D_n \rightarrow \C$ is $\beta$-nilpotent if there is a function $f$ such that
\[
    a(z) = f\bigg(
        \frac{\langle(-4\im(z'),1),v_1\rangle}{2(\im(z_n)-|z'|^2)}, \dots,
        \frac{\langle(-4\im(z'),1),v_m\rangle}{2(\im(z_n)-|z'|^2)}
    \bigg).
\]
In particular, every $\beta$-nilpotent symbol is a function of $(\im(z'), \im(z_n)-|z'|^2)$ and so it is a nilpotent symbol in the sense of \cite{QVUnitBall1}. Hence, we can apply Theorem~10.3 from \cite{QVUnitBall1} to obtain the next result. We have also made some simple changes of variable to achieve our integral expression.

\begin{theorem}\label{thm:specToeplitz-beta-nilpotent}
    Suppose that $G = \R^n$ which acts on $D_n$ by the action $N(n)$ and let $\beta = \{ v_1, \dots, v_m\} \subset \R^n$ be a linearly independent set. Then, for every $\lambda > -1$ there exists a unitary map $R : \cA^2_\lambda(D_n) \rightarrow L^2(\R^{n-1}\times\R_+)$ such that for every essentially bounded $\beta$-nilpotent symbol $a : D_n \rightarrow \C$ of the form
    \[
        a(z) = f\bigg(
            \frac{\langle(-4\im(z'),1),v_1\rangle}{2(\im(z_n)-|z'|^2)}, \dots,
            \frac{\langle(-4\im(z'),1),v_m\rangle}{2(\im(z_n)-|z'|^2)}
        \bigg)
    \]
    we have $R T^{(\lambda)}_a R^* = \gamma_{a,\lambda} I$, a multiplication operator, where $\gamma_{a,\lambda}$ is given by
    \begin{multline*}
        \gamma_{a,\lambda}(y',\xi) =
            \frac{2^{\lambda+1}
            \xi^{\lambda+\frac{n+1}{2}}}{\pi^{\frac{n-1}{2}}\Gamma(\lambda+1)}
                \int_{\R^{n-1}\times\R_+}
                    f\Bigg(
                        \frac{\langle (2x',1), v_1 \rangle}{2 x_n}, \dots
                        \frac{\langle (2x',1), v_m \rangle}{2 x_n}
                    \Bigg) \\
                \times e^{-2\xi x_n - |-\sqrt{\xi}x'+y'|^2} x_n^\lambda
                \dif x' \dif x_n,
    \end{multline*}
    for every $y' \in \R^{n-1}$ and $\xi \in \R_+$.
\end{theorem}

As before, we obtain spectral integral formulas corresponding to Theorem~\ref{thm:specToeplitz-beta-nilpotent} in terms of moment map coordinates for the action $N(n)$ for the nilpotent symbols.

\begin{theorem}\label{thm:specToeplitz-nilpotent}
    For every weight $\lambda > -1$ there exists a unitary transformation $R : \cA^2_\lambda(D_n) \rightarrow L^2(\R^{n-1}\times\R_+)$ such that for every essentially bounded nilpotent symbol $a : D_n \rightarrow \C$ of the form
    \[
        a(z) = f\bigg(
            -\frac{2\im(z_1)}{\im(z_n)-|z'|^2}, \dots, -\frac{2\im(z_{n-1})}{\im(z_n)-|z'|^2},
            \frac{1}{2(\im(z_n)-|z'|^2)}
        \bigg)
    \]
    we have $R T^{(\lambda)}_a R^* = \gamma_{a,\lambda} I$, a multiplication operator, where $\gamma_{a,\lambda}$ is given by
    \[
        \gamma_{a,\lambda}(y',\xi) =
            \frac{\xi^{\lambda+\frac{n+1}{2}}}{2^{n-1} \pi^{\frac{n-1}{2}}\Gamma(\lambda+1)}
                \int_{\R^{n-1}\times\R_+}
                    \frac{f(u) e^{-\frac{\xi}{u_n} -
                        \big|-\frac{\sqrt{\xi} u'}{2u_{n}} + y' \big|^2}
                        \dif u' \dif u_n}{u_n^{\lambda+n+1}} ,
    \]
    for every $y' \in \R^{n-1}$ and $\xi \in \R_+$.
\end{theorem}
\begin{proof}
    We consider the change of coordinates $\varphi : \R^{n-1}\times\R_+ \rightarrow \R^{n-1}\times\R_-$ where
    \[
        u = \varphi(x) =
            \bigg(
                \frac{x'}{x_n}, \frac{1}{2x_n}
            \bigg),
    \]
    which has the inverse
    \[
        \varphi^{-1}(u) = \frac{1}{2u_n} (u',1).
    \]
    For this change of coordinates we also have
    \[
        \dif x' \dif x_n = \frac{\dif u' \dif u_n}{2^n u_n^{n+1}}.
    \]
    Now the integral formula is obtained by using Theorem~\ref{thm:specToeplitz-beta-nilpotent} with the canonical basis $\beta = \{e_1, \dots, e_n \}$ of $\R^n$ and applying the change of coordinates $\varphi$.
\end{proof}

We now provide the corresponding spectral integral formulas for the $\beta$-nilpotent symbols. As before, we keep using the conventions and notation from Remark~\ref{rmk:momentmapfunctions_beta_matrices}.

\begin{corollary}\label{cor:specToeplitz-beta-nilpotent}
    Suppose that $G = \R^n$ which acts on $D_n$ by the action $N(n)$ and let  $\beta = \{v_1, \dots, v_m\} \subset \R^n$ be a linearly independent set. Then, for every $\lambda > -1$ there exists a unitary map $R : \cA^2_\lambda(D_n) \rightarrow \ell^2(\N^{n-1},L^2(\R_+))$ such that for every essentially bounded $\beta$-nilpotent symbol $a : D_n \rightarrow \C$ of the form
    \[
        a(z) = f\bigg(
        \frac{\langle(-4\im(z'),1),v_1\rangle}{2(\im(z_n)-|z'|^2)}, \dots,
        \frac{\langle(-4\im(z'),1),v_m\rangle}{2(\im(z_n)-|z'|^2)}
        \bigg).
    \]
    we have $RT_a^{(\lambda)} R^* = \gamma_{a,\lambda} I$, a multiplication operator, where $\gamma_{a,\lambda}$ is given by
    \[
        \gamma_{a,\lambda}(y',\xi) =
            \frac{\xi^{\lambda+\frac{n+1}{2}}}{2^{n-1} \pi^{\frac{n-1}{2}}\Gamma(\lambda+1)}
                \int_{\R^{n-1}\times\R_+}
                    \frac{f(A(\beta)u^\top) e^{-\frac{\xi}{u_n} - \big|
                        -\frac{\sqrt{\xi} u'}{2u_{n}} + y' \big|^2}
                        \dif u' \dif u_n}{ u_n^{\lambda+n+1}} ,
    \]
    for every $y' \in \R^{n-1}$ and $\xi \in \R_+$.
\end{corollary}
\begin{proof}
    Our $\beta$-nilpotent symbols can be rewritten as
    \[
        a(z) = f\bigg(
        A(\beta)
        \frac{(-4\im(z'),1)^\top}{2(\im(z_n)-|z'|^2)}
        \bigg).
    \]
    Hence, we can apply Theorem~\ref{thm:specToeplitz-nilpotent} to obtain the required expression.
\end{proof}

\subsection{Spectra for $\beta$-quasi-nilpotent symbols.} Let us choose and fix $k \in \Z$ such that $1 \leq k \leq n-1$. By Proposition~\ref{prop:momentmap_betaN(n,k)} a symbol $a : D_n \rightarrow \C$ is $\beta$-quasi-nilpotent if there is some function $f$ such that we can write
\[
    a(z) =
        f\bigg(
            \frac{\langle (2r_{(1)}^2, -4\im(z_{(2)}),1),
                    v_1 \rangle}{2(\im(z_n)-|z'|^2)}, \dots,
            \frac{\langle (2r_{(1)}^2, -4\im(z_{(2)}),1),
                    v_m \rangle}{2(\im(z_n)-|z'|^2)}
            \bigg),
\]
where we have $z=(z',z_{n})=(z_{(1)},z_{(2)},z_{n})$, and the last decomposition is given by the partition $(k,n-k-1,1)$. In particular, we also have
\begin{align*}
    r_{(1)} &= (r_1, \dots, r_k) = (|z_1|, \dots, |z_k|), \\
    r_{(1)}^2 &= (r_1^2, \dots, r_n^2) = (|z_1|^2, \dots, |z_k|^2),\\
   \im(z_{(2)})&=(\im(z_{k+1}),\ldots,\im(z_{n-1})).
\end{align*}
It follows immediately that if a symbol is $\beta$-quasi-nilpotent, then it depends on $(|z_1|, \dots, |z_k|, \im(z_{(2)}), \im(z_n)-|z_{(2)}|^2)$. In particular, every $\beta$-quasi-nilpotent symbol is quasi-nilpotent in the sense of \cite{QVUnitBall1}. Hence, we can apply Theorem~10.4 from \cite{QVUnitBall1} to obtain the next result.

\begin{theorem}\label{thm:specToeplitz-beta-quasi-nilpotent}
    Let $k$ be an integer such that $1 \leq k \leq n-1$ and suppose that $G = \T^k\times\R^{n-k-1}\times\R_+$ which acts on $D_n$ by the action $N(n,k)$. Let $\beta = \{v_1, \dots, v_m\} \subset \R^n$ be a linearly independent set. Then, for every $\lambda > -1$ there exists a unitary map $R : \cA^2_\lambda(D_n) \rightarrow \ell^2(\N^k,L^2(\R^{n-k-1}\times\R_+))$ such that for every essentially bounded $\beta$-nilpotent symbol $a : D_n \rightarrow \C$ of the form
    \[
        a(z) =
            f\bigg(
                \frac{\langle (2r_{(1)}^2, -4\im(z_{(2)}),1),
                        v_1 \rangle}{2(\im(z_n)-|z'|^2)}, \dots,
                \frac{\langle (2r_{(1)}^2, -4\im(z_{(2)}),1),
                        v_m \rangle}{2(\im(z_n)-|z'|^2)}
                \bigg)
    \]
    we have $RT^{(\lambda)}_aR^* = \gamma_{a,\lambda}I$, a multiplication operator, where $\gamma_{a,\lambda}$ is given by
    \begin{multline*}
        \gamma_{a,\lambda}(p,y',\xi) = \\
            = \frac{2^{\lambda+|p|+k+1} \xi^{\lambda+|p|+\frac{n+k+1}{2}}}{ \pi^{\frac{n-k-1}{2}} p! \Gamma(\lambda+1)}
            \int_{\R^k_+\times\R^{n-k-1}\times\R_+}
                f\bigg(\bigg(
                \frac{\langle (2r, 2x',1),
                    v_j \rangle}{2x_n}
                \bigg)_{j=1}^m\bigg) \\
        \times r^p
            e^{-2\xi(x_n+|r|)- |-\sqrt{\xi}x'+y'|^2}
               x_n^\lambda  \dif r \dif x' \dif x_n
    \end{multline*}
    for every $p \in \N^k$, $y'\in \R^{n-k-1}$ and $\xi \in \R_+$.
\end{theorem}

We now obtain the spectral integral formulas corresponding to Theorem~\ref{thm:specToeplitz-beta-quasi-nilpotent} in terms of moment map coordinates for the action of $N(n,k)$ for the quasi-nilpotent symbols.

\begin{theorem}\label{thm:specToeplitz-quasi-nilpotent}
    For every weight $\lambda > -1$ there exists a unitary transformation $R : \cA^2_\lambda(D_n) \rightarrow L^2(\R^{n-1}\times\R_+)$ such that for every essentially bounded nilpotent symbol $a : D_n \rightarrow \C$ of the form
    \[
        a(z) =
        f\bigg(
            \frac{(2r_{(1)}^2, -4\im(z_{(2)}),1)}{2(\im(z_n)-|z'|^2)}
            \bigg),
    \]
    we have $R T^{(\lambda)}_a R^* = \gamma_{a,\lambda} I$, a multiplication operator, where $\gamma_{a,\lambda}$ is given by
    \begin{multline*}
        \gamma_{a,\lambda}(p,y',\xi) =
        \frac{\xi^{\lambda+|p|+\frac{n+k+1}{2}}  }{2^{n-k-1} \pi^{\frac{n-k-1}{2}} p! \Gamma(\lambda+1)} \\
          \int_{\R^k_+\times\R^{n-k-1}\times\R_+}
            \frac{f(u) u_{(1)}^{p}
                e^{-\frac{\xi}{u_n}\big(1 + |u_{(1)}|\big)
                - \big|-\frac{\sqrt{\xi} u_{(2)}}{2u_{n}} + y'\big|^2} \dif u}{ u_n^{\lambda+|p|+n+1}}
    \end{multline*}
    for every $p \in \N^k$, $y'\in \R^{n-k-1}$ and $\xi \in \R_+$.
\end{theorem}
\begin{proof}
    This result is obtained from Theorem~\ref{thm:specToeplitz-beta-quasi-nilpotent} for the canonical basis $\beta = \{e_1, \dots, e_n\}$ and the use of the change of variables from the proof of Theorem~\ref{thm:specToeplitz-nilpotent}.
\end{proof}

Finally, we provide the spectral integral formulas for the $\beta$-nilpotent symbols using moment man coordinates. Recall the standing use of the conventions and notation from Remark~\ref{rmk:momentmapfunctions_beta_matrices}.

\begin{corollary}\label{cor:specToeplitz-beta-nilpotent}
    Suppose that $G = \R^{n}$ which acts on $D_n$ by the action $N(n,k)$ and let  $\beta = \{v_1, \dots, v_m\} \subset \R^n$ be a linearly independent set. Then, for every $\lambda > -1$ there exists a unitary map $R : \cA^2_\lambda(D_n) \rightarrow \ell^2(\N^{n-1},L^2(\R_+))$ such that for every essentially bounded $\beta$-nilpotent symbol $a : D_n \rightarrow \C$ of the form
    \[
        a(z) =
            f\bigg(
                \frac{\langle (2r_{(1)}^2, -4\im(z_{(2)}),1),
                        v_1 \rangle}{2(\im(z_n)-|z'|^2)}, \dots,
                \frac{\langle (2r_{(1)}^2, -4\im(z_{(2)}),1),
                        v_m \rangle}{2(\im(z_n)-|z'|^2)}
                \bigg),
    \]
    we have $RT_a^{(\lambda)} R^* = \gamma_{a,\lambda} I$, a multiplication operator, where $\gamma_{a,\lambda}$ is given by
    \begin{multline*}
        \gamma_{a,\lambda}(p,y',\xi) =
        \frac{\xi^{\lambda+|p|+\frac{n+k+1}{2}}  }{2^{n-k-1} \pi^{\frac{n-k-1}{2}} p! \Gamma(\lambda+1)} \\
          \int_{\R^k_+\times\R^{n-k-1}\times\R_+}
            \frac{f(A(\beta)u^\top) u_{(1)}^{p}
                e^{-\frac{\xi}{u_n}\big(1 + |u_{(1)}|\big)
                - \big|-\frac{\sqrt{\xi} u_{(2)}}{2u_{n}} + y'\big|^2} \dif u}{ u_n^{\lambda+|p|+n+1}}
    \end{multline*}
    for every $p \in \N^k$, $y'\in \R^{n-k-1}$ and $\xi \in \R_+$.
\end{corollary}
\begin{proof}
    Our $\beta$-quasi-nilpotent symbols can be written as
    \[
        a(z) =
        f\bigg(
            A(\beta)\frac{(2r_{(1)}^2, -4\im(z_{(2)}),1)^\top}{2(\im(z_n)-|z'|^2)}
        \bigg),
    \]
    and so the result follows from Theorem~\ref{thm:specToeplitz-quasi-nilpotent}.
\end{proof}

\end{document}